\documentclass[12pt,reqno]{amsproc}

\usepackage[margin=1in]{geometry}
\usepackage[T1]{fontenc}
\usepackage{bbm}
\usepackage[usenames]{color}
\usepackage{mathpazo}
\usepackage{amsmath,amssymb,amsthm}
\usepackage{wrapfig}
\usepackage{tikz-cd}
\usepackage[shortlabels]{enumitem}

\usepackage{caption}
\usepackage{mathrsfs,bm,nicefrac}
\usepackage[all,cmtip]{xy}

\definecolor{PineGreen}{rgb}{0.0,0.47,0.44}
\definecolor{MidnightBlue}{rgb}{0.1,0.1,0.44}
\definecolor{magenta}{rgb}{1.0,0.0,1.0}
\definecolor{bl1}{HTML}{4479A1}
\definecolor{pur1}{HTML}{52196D}
\definecolor{mag1}{HTML}{2AD0F1}
\definecolor{org1}{rgb}{.92,.39.21}
\definecolor{pur2}{rgb}{.53,.47,.7}

\usepackage{hyperref}
\hypersetup{
	colorlinks=true,
	linkcolor=blue,
	citecolor=PineGreen,
	filecolor=magenta,
	urlcolor=blue
}

\newcommand{\spc}{\hspace*{0.25in}}

\makeatletter
\newcommand{\eqnum}{\refstepcounter{equation}\textup{\tagform@{\theequation}}}
\makeatother

\newtheorem{theorem}{Theorem}
\numberwithin{theorem}{section}
\newtheorem{proposition}[theorem]{Proposition}
\newtheorem*{theorem*}{Theorem}
\newtheorem{lemma}[theorem]{Lemma}
\newtheorem{corollary}[theorem]{Corollary}
\theoremstyle{definition}

\newtheorem{definition}[theorem]{Definition}

\theoremstyle{remark}
\newtheorem{remark}[theorem]{Remark}
\newtheorem{example}[theorem]{Example}

\newcommand{\Con}{\mathbf{Con}}
\newcommand{\Gr}{\mathbf{Gr}}
\newcommand{\Ide}{\mathscr{I}}

\newcommand{\RR}{\mathbb{R}}

\newcommand{\PP}{\mathbb{P}}

\newcommand{\pp}{\mathbb{P}}
\newcommand{\CC}{\mathbb{C}}

\DeclareMathOperator{\codim}{codim}

\newcommand{\bF}{\mathbf{F}}

\newcommand{\bT}{\mathbf{T}}
\newcommand{\bB}{\mathbf{B}}

\newcommand{\cP}{\mathscr{P}}

\newcommand{\ass}{\text{\rm Assoc}}

\newcommand{\bV}{\mathbf{V}}

\newcommand{\fS}{\mathfrak{S}}

\newcommand{\set}[1]{{\left\{{#1}\right\}}}
\newcommand{\ip}[1]{\left\langle{#1}\right\rangle}

\makeatletter
\DeclareRobustCommand
\myvdots{\vbox{\baselineskip4\p@ \lineskiplimit\z@
		\hbox{.}\hbox{.}\hbox{.}}}
\makeatother

\newcommand{\surj}{\twoheadrightarrow}
\usepackage[foot]{amsaddr}

\def\DD{D\kern-.7em\raise0.3ex\hbox{\char '55}\kern.33em}

\usepackage{booktabs}
\usepackage{colortbl}
\definecolor{Ftitle}{RGB}{11,46,108}
\colorlet{tableheadcolor}{Ftitle!25} 
\colorlet{tablerowcolor}{gray!10} 
\newcommand{\Pure}{\text{\rm Pure}}
\begin{document}
	
	\title{Conormal Spaces and Whitney Stratifications}
	\author{Martin Helmer} 
	\address[MH]{
		Department of Mathematics, North Carolina State University,
		Raleigh, NC, USA}\email{mhelmer@ncsu.edu}
	\author{Vidit Nanda}\address[VN]{Mathematical Institute,
		University of Oxford, Oxford, UK}\email{nanda@maths.ox.ac.uk}
	
	\maketitle

	\begin{abstract}
		We describe a new algorithm for computing Whitney stratifications of complex projective varieties. The main ingredients are (a) an algebraic criterion, due to L\^e and Teissier, which reformulates Whitney regularity in terms of conormal spaces and maps, and (b) a new interpretation of this conormal criterion via primary decomposition, which can be practically implemented on a computer. We show that this algorithm improves upon the existing state of the art by several orders of magnitude, even for relatively small input varieties. En route, we introduce related algorithms for efficiently stratifying affine varieties, flags on a given variety, and algebraic maps.\vspace{3mm}\newline
	\vspace{3mm}{\bf Mathematics Subject Classification: }14B05, 14Q20,  32S15, 32S60.\newline
	\end{abstract}

	\section*{Introduction}
	
	The quest to define and study singular spaces counts among the most spectacular success stories of twentieth century mathematics. Much of the underlying motivation arose from algebraic geometry, where the spaces of interest -- namely, the vanishing loci of polynomials -- contain singular points even in the simplest of cases. Without the benefit of hindsight, it remains a Herculean task to construct good models of singular spaces that are simultaneously broad enough to  include all analytic varieties and narrow enough to exclude various pathological spaces which arise as zero sets of arbitrary smooth functions. The standard solution to this conundrum, which we describe below, was first proposed by Whitney \cite{whitney1965tangents} and subsequently refined by Thom \cite{thom, thom2}, Mather \cite{Mather2012}, Goresky-MacPherson \cite{IH, IH2, SMTbook}, L\^{e}-Teissier \cite{le1988limites}, Fulton \cite{fulton2013intersection}, Cappell-Shaneson \cite{Cappell1991stratifiable}, Weinberger \cite{weinberger} and others. 
	
	\subsection*{Whitney's Condition (B)} As a natural starting point, one can at least require each candidate space $X$ under consideration to embed in some Euclidean space $\RR^n$ and to admit a partition into smooth submanifolds, say
	\[
	X = \coprod_i M_i,
	\] 
	with $\dim M_i = i$. An entirely reasonable first attempt at constructing such $M_i$ from $X$ might proceed as follows. For each dimension $0 \leq i \leq n$ and subset $Y \subset \RR^n$, let $\mathscr{E}_i(Y)$ denote the set of points $p$ in $Y$ which admit an open neighbourhood $U_p \subset Y$ homeomorphic to $\RR^i$. Then, recursively define
	\begin{align*}
	M_n &:= \mathscr{E}_n(X), \text{ and } \\
	M_i &:= \mathscr{E}_i(X-M_{>i}) \text{ for }0 \leq i < n,
	\end{align*}
	where $M_{>i}$ denotes the union $\bigcup_{j>i}M_j$. Unfortunately, this recursive strategy does not produce a desirable partition. Perhaps the simplest way to see the underlying problem is to try constructing these $M_i$ by hand when $X \subset \RR^3$ is the singular surface depicted in Figure \ref{fig:whitcusp}. 

\begin{figure}[h!]
	\begin{picture}(550,120)
\put(45,0){\includegraphics[width=.8\linewidth]{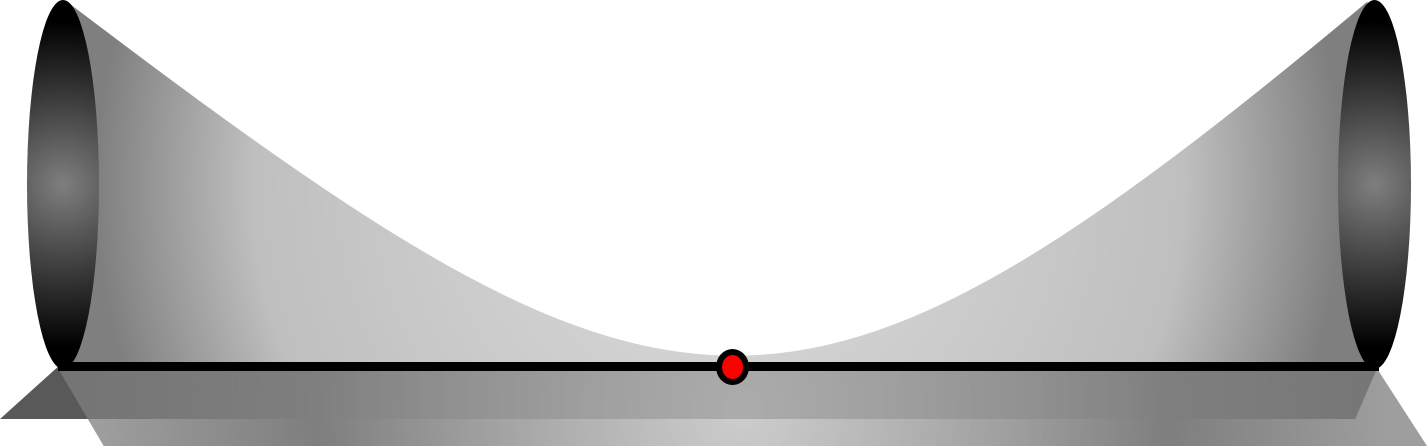}}
\put(235,31){$0$}
\end{picture}
\caption{The {\em Whitney cusp} depicted above is the hypersurface of $\RR^3$ given by $y^2+z^3-x^2z^2=0$. The entire $x$-axis, drawn horizontally, is singular.} \label{fig:whitcusp}
\end{figure}

	Since $M_3$ is empty, one identifies $M_2 \subset X$ as the set of points with two-dimensional Euclidean neighbourhoods; and upon removing these, only the $x$-axis remains. This axis must therefore equal $M_1$, and we obtain a partition of $X$ into one and two-dimensional smooth manifolds. The issue here is that the origin has a singularity type which is different from all other points lying on $M_1$ --- a small neighbourhood in $X$ around the origin is not homeomorphic to a small neighbourhood around any other point lying on the $x$-axis. More precisely, let $G$ be the group of homeomorphisms $f:X \to X$ so that $f$ is isotopic to the identity and its restriction to each $M_i$ is a diffeomorphism. It turns out that $G$ acts transitively on the two connected components of $M_1-\set{0}$ while fixing $0$ itself. Thus, our recursive strategy must be amended so that the $M_i$ are $G$-{\em equisingular} in this sense, which would automatically separate $0$ from $M_1$ into a separate stratum.
	
	Whitney's ingenious approach from \cite{whitney1965tangents} was to consider the behaviour of limiting tangent spaces as one approaches a point in some $Y := M_i$ in two different ways: one in a tangential direction along $Y$ itself, and another in a normal direction along some other $X := M_j$ for $j > i$. Let $\set{x_i}$ and $\set{y_i}$ be sequences of points in $X$ and $Y$ respectively which both converge to the same $y$ in $Y$. Write $T_i$ for the tangent space of $X$ at $x_i$ and $\ell_i$ for the secant line $[x_i,y_i]$ joining each $x_i$ to the corresponding $y_i$ in the ambient $\RR^n$. The pair $(X,Y)$ is said to satisfy Whitney's {\bf Condition (B)} if the limiting tangent space $T = \lim T_i$ contains the limiting secant line $\ell = \lim \ell_i$ whenever both limits exist. In our example, one can find sequences of points
	in $X := M_2$ and $Y := M_1$, both limiting to the problematic point $0$, for which $\ell \not\subset T$ --- these are illustrated in Figure \ref{fig:condB}.

\begin{figure}	
		\begin{picture}(550,130)
\put(45,0){\includegraphics[width=.8\linewidth]{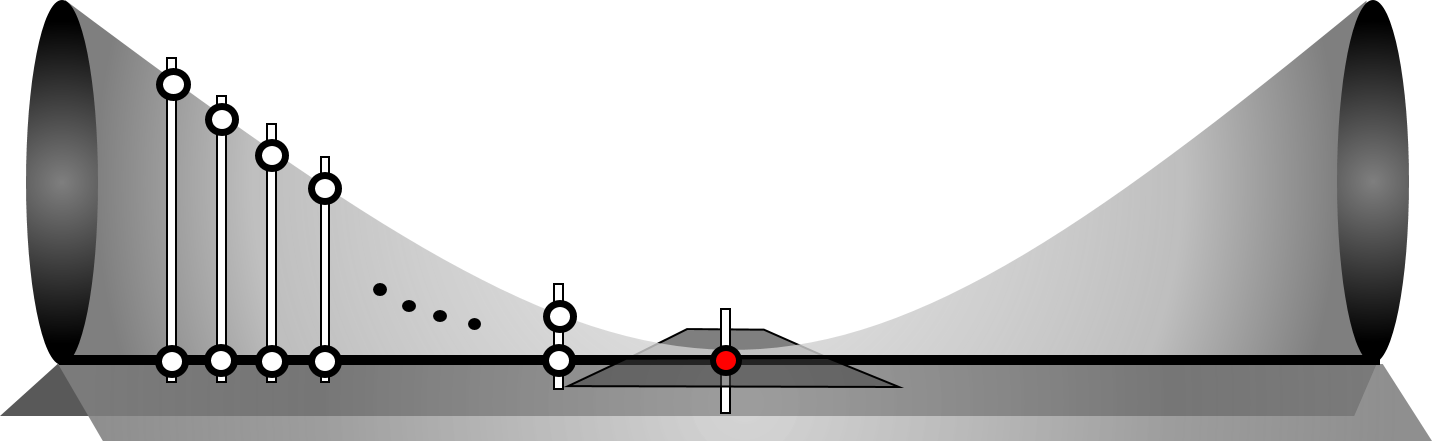}}
\put(231,41){$\ell$}
\put(90,-4){\vector(0,1){17}}
\put(40,60){\vector(1,0){45}}
\put(29,57){$\ell_i$}
\put(263,46){\vector(-1,-1){15}}
\put(263,46){$T$}
\put(88,-12){$y_i$}
\put(88,107){$x_i$}
\end{picture}\vspace{10mm}
\caption{Sequences of points $\set{x_i} \subset M_2$ and $\set{y_i} \subset M_1$ which both converge to the origin. The limiting tangent plane $T$ of the $x_i$'s is horizontal, while the limiting secant line $\ell$ of the $[x_i,y_i]$'s is vertical. Thus, Condition (B) would be violated if the origin was included in $M_1$.}
\label{fig:condB}
\end{figure}
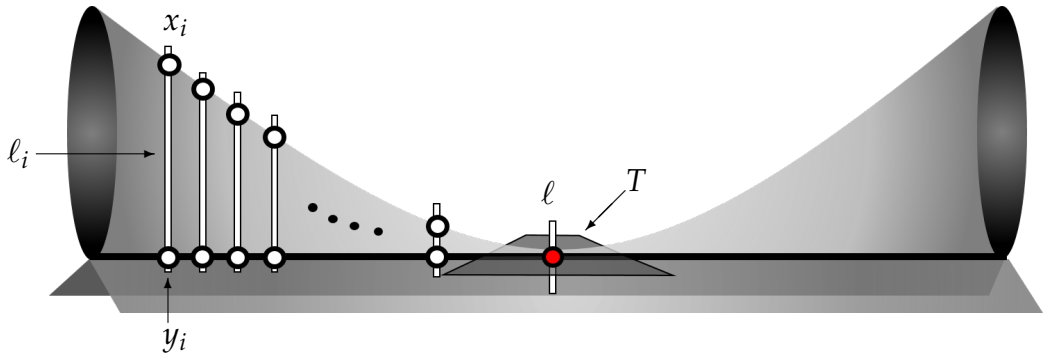
	
	It is a foundational result in stratification theory that equisingularity is satisfied by any decomposition $X = \coprod_i M_i$ for which all pairs $(M_j,M_i)$ satisfy Whitney's Condition (B) --- see \cite{Mather2012} or \cite[Part I Ch 1.5]{SMTbook}. Since their very inception, such {\em Whitney stratifications} have been used to define and compute myriad important algebraic-topological invariants of singular spaces and related structures, even in cases where the invariants do not ultimately depend on the chosen stratification. Prominent examples include intersection homology groups \cite{IH, IH2}, stratified vector fields \cite{brasseletSeadeSuwa}, characteristic varieties \cite{Ginsburg1986}, Euler obstructions \cite{gonzalez1981obstruction,  rodriguez2017computing} and Chern classes \cite{MacPherson1974}, to name but a few.
	
	\subsection*{Conormal Spaces}
	
	Given a $k$-dimensional projective variety $X \subset \pp^n$, let $X_\text{reg}$ be  the smooth locus of $X$ and note that there is a well-defined tangent space $T_xX_\text{reg}$ at each point $x$ in $X_\text{reg}$. This tangent space naturally resides in the Grassmannian $\Gr(k,n+1)$ of $k$-dimensional subspaces of $\CC^{n+1}$; let us consider the map 
	\[
	\tau:X_\text{reg} \to X \times \Gr(k,n+1)
	\] that sends each $x$ to the pair $(x,T_xX_\text{reg})$. Thus, $\tau$'s image is the graph of the Gauss map of $X$; taking the closure of this image creates a new (usually singular) space ${\bf N}(X)$, called the {\bf Nash blowup} of $X$ (see \cite[Sec 16]{whitney1965tangents} or \cite{MacPherson1974}). The fiber over each point $x \in X$ of the evident projection map ${\bf N}(X) \surj X$ catalogues all the limiting tangent spaces at $x$; two such fibers are illustrated in Figure \ref{fig:nash}.
	
	\begin{figure}[h!]
		\includegraphics[width = .7\linewidth]{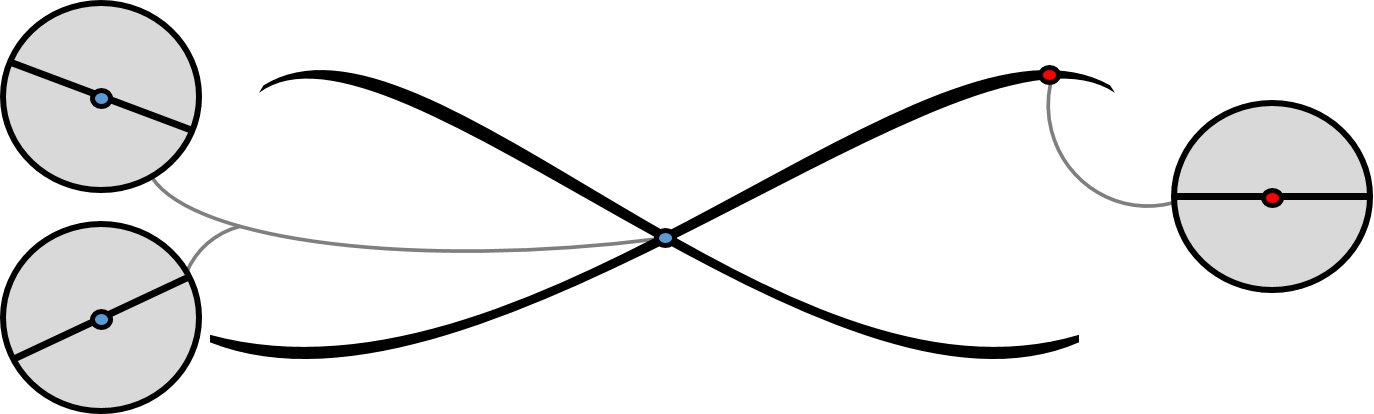}
		\caption{The fibers of the Nash blowup are depicted over a smooth (red) and singular (blue) point of an underlying curve.} \label{fig:nash}
	\end{figure}
	
	From the perspective of analysing limiting tangent spaces, the Nash blowup is an optimal object. However, the geometry of the Grassmannian is rather intricate, and this makes it difficult to explicitly compute defining equations of ${\bf N}(X)$ from those of $X$. One remedy is to systematically replace $\Gr(k,n+1)$ with the dual projective space $(\pp^n)^*$ in the construction described above. With this modification in place, we consider the set of all hyperplanes in $\pp^{n}$ -- or equivalently, points in the dual projective space $(\pp^n)^*$ -- that contain the tangent space at each $x$ in $X_\text{reg}$. Passing to the closure in $X \times (\pp^n)^*$ produces the {\bf conormal space} $\Con(X)$ of $X$, and the natural projection $\kappa_X:\Con(X) \surj X$ is called the {\bf conormal map}. The conormal space retains essential information about limiting tangents; and crucially, $\Con(X)$ is also a projective variety whose defining equations can be easily extracted from those of $X$.

	\subsection*{This Paper} Here we introduce an algorithm for building Whitney stratifications of complex projective varieties. The main reason for restricting our focus to such varieties rather than the far more general class of real semialgebraic sets (which are equally easy to represent on a computer) is that the ingredients required to implement our algorithm are only available in the complex algebraic setting. In any event, the immediate obstacle is that it appears hopeless to try verifying Condition (B) directly for a pair of smooth  quasiprojective varieties, since this would require computation of limiting tangent planes and secant lines over arbitrary pairs of infinite sequences. Attempting to bypass this problem by constructing the Nash blowup also appears to be a daunting task. Thus, we turn to the conormal space $\Con(X)$ of the given input variety $X$.
	
	The good news comes in the form of a result by L\^{e} and Teissier, which  provides a complete characterisation of Condition (B) for projective varieties in terms of their conormal spaces \cite[Proposition 1.3.8]{le1988limites}. This criterion asserts that for any projective variety $X \subset \pp^n$ and smooth quasiprojective subvariety $Y \subset X$, the pair $(X_\text{reg},Y)$ satisfies Condition (B) if and only if the associated primes of the integral closure of the ideal sheaf of $\kappa_X^{-1}(Y)$ are contained among the associated primes of the ideal sheaf of $$\Con(X) \cap \Con(Y)\cap \overline{\kappa_X^{-1}(Y)}.$$ There are now two caveats to consider --- first, as remarked above, we are not aware of any analogous criterion for pairs of smooth real (semi)algebraic sets. Second, and more serious from our perspective, is the fact that even the most basic algorithmic tasks involving integral closures are computationally prohibitive \cite[Sec 9.3]{vasconcelos2006integral}. 
	
	Our main result circumvents the latter issue by making use of a carefully chosen primary decomposition. 
	Let $R$ be the coordinate ring of $\pp^n \times (\pp^n)^*$, and consider a projective variety $X \subset \pp^n$ with conormal map $\kappa_X:\Con(X) \surj X$. For any subvariety $Y \subset X$, the ideal sheaf of $\kappa_X^{-1}(Y)$ constitutes an honest homogeneous ideal in $R$ and is given by $I_Y + I_{\Con(X)}$, and it therefore makes sense from an algorithmic perspective to perform a primary decomposition \cite[Chapter 4.8]{CLO} of this ideal. Given such a homogeneous ideal $I \lhd R$, we will denote the associated projective subvariety of $\pp^n \times (\pp^n)^*$ by $\bV(I)$. With this preamble in place, we now state our main theorem.\newpage
	
	\begin{theorem*}	Let $X\subset \pp^n$ be a pure dimensional projective variety and $Y$ a nonempty irreducible subvariety of its singular locus $X_{\rm sing}$. Consider a primary decomposition (of $R$-ideals)
$
{\Ide}[{\kappa_X^{-1}(Y)}]= \bigcap_{i=1}^s Q_i,
$
and let $\sigma \subset \set{1,2,\ldots,s}$ be the set of indices $i$ for which $\dim \kappa_X(\bV(Q_i)) < \dim Y$. Define 
\begin{align*}\label{eq:W}
A := \left[\bigcup_{i \in \sigma}\kappa_X(\bV(Q_i))\right] \cup Y_{\rm sing}.
\end{align*} 
Then the pair $(X_\text{\rm reg},Y-A)$ satisfies Condition (B). 
\end{theorem*}
	
	
	Computing a primary decomposition of an ideal reduces to standard Gr\"obner basis calculations \cite{decker1999primary,gianni1988grobner} and is widely implemented in computer algebra systems, so
	this result directly leads to our recursive algorithm for stratifying complex projective varieties. Before describing the details, we highlight three relevant features. First, given a $k$-dimensional input variety $X \subset \pp^n$, the output Whitney stratification is produced in the form of (defining equations for) a nested sequence $X_\bullet$ of subvarieties 
	\[
	X_0 \subset X_1 \subset \cdots \subset X_k = X,
	\]
	so that the desired manifold partition $X = \coprod_i M_i$ is given by $M_i := X_i - X_{i-1}$. Second, this algorithm can easily be used to produce Whitney stratifications of affine complex varieties as well --- first pass to the projective closure $X$, build its Whitney stratification $X_\bullet$, then dehomogenize the resulting $X_i$'s.  And third, given any flag $\bF_\bullet$ of projective subvarieties:
	\[
	\bF_0X \subset \bF_1X \subset \cdots \subset \bF_\ell X = X,
	\]
	our algorithm can be adapted so that its output is subordinate to this flag. In other words, we can guarantee that each connected component of a given $M_i$ lies in a single successive difference $\bF_jX - \bF_{j-1}X$ (see Section \ref{sec:flagstrat}). As a consequence, we are also able in Section \ref{sec:mapstrat} to stratify generic morphisms of projective varieties by availing of their Thom-Boardman flags \cite{boardman, thom}.
	
	\subsection*{Related Work} The current state of the art in this area appears to be the recent work of {\DH}inh and Jelonek \cite{dhinh2019thom}. To the best of our knowledge, given a complex algebraic variety $X$ embedded in $n$-dimensional affine or projective space, all prior stratification methods (such as \cite{mostowski1991complexity, rannou1998complexity} for instance) require quantifier elimination in approximately $4n$ real variables.  While critical points based algorithms for quantifier elimination, such as that of Grigoriev and Vorobjov \cite{grigor1988solving}, have existed in the  theoretical literature for many years these remain unimplemented (see the discussion in Subsection \ref{subsec:comparison_theory}). Therefore, quantifier elimination is generally accomplished using some version of the cylindrical algebraic decomposition algorithm of Collins \cite{arnon1984cylindrical,Basu2006Algorithms,caviness2012quantifier}, which is known to be extremely difficult in practice. For stratifications of real varieties given by unions of transversely-intersecting {\em smooth} subvarieties, the theoretical complexity of quantifier elimination can be somewhat improved \cite{vorobjov1992effective}. As far as we are aware, no implementation of any such quantifier elimination based Whitney stratification algorithms has ever been produced. 
	
	The work of {\DD}inh and Jelonek \cite{dhinh2019thom} improves upon such cylindrical algebraic decomposition (CAD) approaches by requiring only Gr\"obner basis (GB) computations in approximately $4n$ complex variables. The advantage enjoyed by GB methods over their CAD counterparts in practice has been well-documented \cite{england2016experience,wilson2012speeding}. In fact, it is remarked in \cite[III.D]{england2016experience} that 
	\begin{quote}
		{\em``Although like CAD
			the calculation of GB is doubly exponential in the worst case, GB computation is now mostly trivial for any problem on which CAD construction is tractable."}
	\end{quote} In fairness, it should be noted that cylindrical algebraic decompositions apply to a much wider class of singular spaces (i.e., real semialgebraic sets), where Gr\"obner basis methods are entirely unavailable.
	
	Our algorithm further reduces the Whitney stratification problem to Gr\"obner basis type computations in approximately $2n$ complex variables, and additionally, is able to preserve the sparsity structure of the input in ways that make a significant difference to real-world performance. By contrast, the algorithm of \cite{dhinh2019thom} requires various choices of generic linear forms, which end up removing a lot of sparsity from the systems being considered. As an added bonus our algorithm is deterministic, while the algorithm of \cite{dhinh2019thom} is probabilistic. We have implemented both algorithms, and provide a performance comparison in the final Section of this paper.

	\subsection*{Outline} In Sections \ref{sec:conormal} and \ref{sec:whitstrat}, we briefly review conormal spaces and Whitney stratifications respectively, with a view towards describing the L\^{e}-Teissier criterion for Whitney's Condition (B). Section \ref{sec:saturations}, which is focused on the primary decomposition approach to this criterion, forms the technical heart of this paper and contains a proof of our main result. We describe our recursive algorithm for constructing Whitney stratifications of complex projective varieties in Section \ref{sec:WSalgorithm} and verify its correctness. In Section \ref{sec:flagstrat} we modify this algorithm to produce flag-subordinate stratifications, which are then used to stratify projective morphisms in Section \ref{sec:mapstrat}. Finally, Section \ref{sec:runtimes} provides both complexity estimates and empirical evidence that the algorithm described here readily outperforms the existing state of the art.

{\footnotesize	
	\subsection*{Acknowledgements} VN is grateful to Mark Goresky and Robert MacPherson for several valuable discussions on the local geometry of singular spaces; his work was supported by EPSRC grant EP/R018472/1 and by DSTL grant D015 funded through the Alan Turing Institute. We thank the anonymous referee for their careful reading of the paper as well as their helpful comments and corrections.  We are particularly grateful  to the research group comprising Lihong Zhi, Nan Li, Zhihong Yang and Zijia Li for alerting us to an error in the published version of this paper \cite{hn}. They discovered that the accompanying software package did not perform correctly on the projective Whitney Cusp. This error has been corrected in this version of the paper and in the new version of the code. Finally, we would like to thank Bernard Teissier for alerting us to a mistake in our description of the integral closure of an ideal in an earlier version of this manuscript.}

	\section{Conormal Spaces}\label{sec:conormal}
	
	For each pair of positive integers $n$ and $k$ with $n+1 \geq k$, let $\Gr(k,n+1)$ denote the complex Grassmannian whose points correspond to all the  $k$-dimensional linear subspaces of $(n+1)$-dimensional affine space $\CC^{n+1}$. The usual projective space $\pp^n$ equals $\Gr(1,n+1)$ while its dual $(\pp^n)^*$ is $\Gr(n,n+1)$. Let $X \subset \pp^n$ be a connected $k$-dimensional complex analytic space whose regular and singular loci will be written $X_\text{reg}$ and $X_\text{sing} := (X - X_\text{reg})$ respectively. Recall that the tangent space to the smooth manifold $X_\text{reg}$ at a given point $x$ is a subspace $T_xX_\text{reg}$ in $\Gr(k,n+1)$.

	\begin{definition} \label{def:conormal} The {\bf conormal space} of $X$ is the subset of $\pp^n \times (\pp^n)^*$ determined by the closure
		\[
		\Con(X)=\overline{\set{(x,\xi) \mid x \in X_{\rm reg} \text{ and } T_xX_{\rm reg} \subset \xi}}.
		\]
		Thus, a point $(x,\xi)$ in $\pp^n \times (\pp^n)^*$ lies in $\Con(X)$ if and only if there exists a sequence of points $\set{x_i} \subset X_\text{reg}$ which converge to $x$ in $\pp^n$ and a sequence of hyperplanes $\set{\xi_i} \subset (\pp^n)^*$ which converge to $\xi$ in $(\pp^n)^*$ so that $T_{x_i}X_\text{reg} \subset \xi_i$ holds for all $i \gg 1$. 
	\end{definition}
	
	The map $\kappa_X:\Con(X) \to X$ induced by the evident projection $(x,\xi) \mapsto x$ is called the {\bf conormal map} of $X$; the fiber $\kappa_X^{-1}(x)$ of $\kappa_X$ over a point $x$ in $X$ is the set of all hyperplanes in $(\pp^n)^*$ which contain a limiting tangent space at $x$. Conormal spaces and maps are well-studied classical objects of substantial interest in complex geometry, with deep connections to polar varieties, Nash blowups, microlocal analysis and beyond \cite{FTpolar}. Here we will be interested exclusively in conormal maps of complex projective subvarieties $X \subset \pp^n$ --- in this special case, $\Con(X)$ is an $n$-dimensional complex subvariety of $X \times (\pp^n)^*$, the conormal map $\kappa_X$ is algebraic, and the dimension of each fiber $\kappa_X^{-1}(x)$ is no larger than $n$ (see \cite[Proposition 2.9]{FTpolar} and references therein).

	\begin{proposition}
		Let $X \subset \pp^n$ be a pure dimensional complex variety. If $X' \subset X$ is any Zariski-dense subset, then $\Con(X)=\Con(X')$. \label{prop:ConormalOfDense}
	\end{proposition}
	\begin{proof}
		Since $X'$ is dense in the closed connected algebraic subspace $X \subset \pp^n$, for any point $(x,\xi)\in \Con(X)$ there exists a sequence of points $\{x_i\}$ in $X'_\text{reg}$ converging to $x$ and an induced sequence of hyperplanes $\xi_i$ containing $T_{x_i}X'_\text{reg}$ which converge to  $\xi$. Since $\Con(X)$ is closed by definition, we have $\Con(X)=\Con(X')$. 
	\end{proof}
	
	\section{Whitney Stratifications}\label{sec:whitstrat}
	
	Let $X, Y \subset \CC^{n+1}$ be smooth complex manifolds with $\dim(Y)<\dim(X)$. A point $p$ in $Y$ is said to satisfy Whitney's {\bf Condition (B)} with respect to $X$ if the following property \cite[Sec 19]{whitney1965tangents} holds:
	\begin{quote} for any sequences of points $\set{x_i} \subset X$ and $\set{y_i} \subset Y$ both converging to $p$, if the secant lines $\ell_i = [x_i,y_i]$ converge to some limiting line $\ell$ in $\pp^n$ and if the tangent spaces $T_{x_i}X$ converge to some limiting plane $T$ in the Grassmannian $\Gr(\dim X,n+1)$, then $\ell \subset T$.
	\end{quote}
	More generally, we say that the pair $(X,Y)$ satisfies Condition (B) if the above property holds for every point $p$ in $Y$. Note that $(X,Y)$ vacuously satisfies Condition (B) if the closures $\overline{X}$ and $\overline{Y}$ do not intersect in $\CC^{n+1}$. The result below is also elementary, and we have only highlighted it here since we appeal to it rather frequently.
	\begin{proposition} \label{prop:condBdense}
		Assume that a pair $(X,Y)$ of smooth complex manifolds satisfies Condition (B). If $X' \subset X$ is a dense submanifold of $X$ and $Y' \subset Y$ an arbitrary submanifold of $Y$, then $(X',Y')$ also satisfies Condition (B).
	\end{proposition}
	\begin{proof}
		Since $X' \subset X$ and $Y' \subset Y$, every sequence $(\set{x_i},\set{y_i}) \subset X' \times Y'$ is automatically a sequence in $X \times Y$. And since $X'$ is dense in $X$, we have an equality $T_{x_i}X' = T_{x_i}X$ of tangent planes in the appropriate Grassmannian. Thus, $(X',Y')$ satisfies Condition (B) because $(X,Y)$ does.
	\end{proof}
	
	Condition (B) serves as a regularity axiom which can be used to induce a particularly well-behaved and useful class of decompositions of analytic spaces into submanifolds.
	
	\begin{definition}\label{def:whitstrat}
		A ($k$-dimensional) {\bf Whitney stratification} of a complex analytic subspace $W \subset \CC^{n+1}$ is a filtration $W_\bullet$ by closed subsets
		\[
		\varnothing = W_{-1} \subset W_0 \subset W_1 \subset \cdots \subset W_{k-1} \subset W_k = W,
		\]
		where each difference $M_i := W_i-W_{i-1}$ is a complex analytic $i$-dimensional manifold subject to the following conditions. The connected components of $M_i$, called the $i$-dimensional {\bf strata}, must obey the following axioms:
		\begin{enumerate} 
			\item {\bf local finiteness:} every point $p$ in $W$ admits an open neighbourhood which intersects only finitely many strata; 
			\item {\bf frontier:} for each stratum $S \subset X$, the difference $(\overline{S}-S)$ is a finite union of lower-dimensional strata; and finally,
			\item {\bf condition (B):} each pair of strata (regardless of  dimension) satisfies Condition (B). 
		\end{enumerate}
	\end{definition}
	
	Whitney showed in \cite[Section 19]{whitney1965tangents} that every variety $X$ admits a Whitney stratification $X_\bullet$ for which each constituent $X_i \subset X$ is a subvariety. His original definition from that paper contains an additional Condition (A), which is now known by Mather's work to be superfluous \cite[Proposition 2.4]{Mather2012}. In fact, even the frontier axiom follows from Condition (B) for  locally finite stratifications \cite[Corollary 10.5]{Mather2012}, so Condition (B) will remain our sole focus in this paper. The inherent difficulty here is that verifying Condition (B) involves testing the behavior of limiting tangent spaces and secant lines over arbitrary families of infinite sequences. 
 
 Fortunately, there is a beautiful alternate characterization in terms of conormal maps due to L\^e and Teissier \cite[Proposition 1.3.8]{le1988limites}, which we describe below. The following notation has been employed in its statement: for each closed subscheme $Z$ of $\pp^n$ (or of $\pp^n\times (\pp^n)^*$), the defining sheaf of ideals is written $\Ide[Z]$; note that $\Ide[Z]$ is simply an ideal in the appropriate coordinate ring $R$, however it may fail to be radical. In contrast, the radical ideal associated to a variety $V$ in either of our ambient spaces will be written $I_V$. We also recall that for a polynomial ideal $I \lhd R$ with primary decomposition $I=Q_1 \cap \cdots \cap Q_r$, the set $\ass(I)$ of $I$'s {\bf associated primes}  consists of prime ideals given by taking the radical of each primary component $\{ \sqrt{Q_1}, \dots,\sqrt{Q_r} \}$. 

 \begin{lemma} \label{lemma:directTranslation} Let $X \subset \pp^n$ be a projective variety with conormal map $\kappa_X:\Con(X) \surj X$, and consider a subvariety  $Y$ of $X$ satisfying $Y \subset X_\text{\rm sing}$. Let $I_p$ be the maximal ideal of a point $p\in Y_{\rm reg}$ considered in the coordinate ring of $\pp^n\times (\pp^n)^*$. 
Then Condition (B) holds for the pair $(X_\text{\rm reg},Y_\text{\rm reg})$ at $p$ if and only if we have a containment 
\[
\ass\left(\Ide[\Con(X)\cap \Con(Y)]+I_p\right)\supset \ass\left( \overline{\Ide}[\kappa_X^{-1}(Y)]+I_p\right)
\]
of associated primes.
\end{lemma}

If $A \subset Y$ is a closed proper subvariety that contains $Y_\text{sing}$ but not the point $p$, then the above criterion for the pair $(X_\text{\rm reg},Y_\text{\rm reg})$ at $p$ gives an identical criterion for the pair $(X_\text{\rm reg},Y-A)$ at $p$. Our goal here is to derive from this local result a global algorithmic criterion for checking whether or not $(X_\text{reg},Y-A)$ satisfies Condition (B) at {\em all} points $y \in Y-A$. In the statement below, ${\Ide}[\kappa_X^{-1}(Y-A)]$, respectively $\overline{\Ide}[\kappa_X^{-1}(Y-A)]$, denotes the intersection of all primary components of ${\Ide}[\kappa_X^{-1}(Y)]$, respectively $\overline{\Ide}[\kappa_X^{-1}(Y)]$, which are not supported on $\kappa_X^{-1}(A)$.

\begin{proposition}\label{prop:conormCond}
Let $X, Y$ be as defined in the statement of Lemma \ref{lemma:directTranslation} and let $A\subsetneq Y$ with $Y_{\rm sing }\subset A$. Then Condition (B) holds for all points in the pair $(X_{\rm reg}, Y-A)$ if and only if we have the containment
\begin{equation}
\ass\left(\Ide[\Con(X)\cap \Con(Y)] +  \Ide[\kappa_X^{-1}(Y-A)] \right) \supset \ass\left(\overline{\Ide}[\kappa_X^{-1}(Y-A)]\right)
\label{eq:lemmaOpenAssPrime}
\end{equation}
\label{lemma:correctedVersion}
of associated primes.
	\end{proposition}
\begin{proof}
Set $Y'=Y-A$. We know from Lemma \ref{lemma:directTranslation} that Condition (B) holds for $(X_\text{reg},Y')$ at a point $p\in Y'$ if and only if we have the containment 
\begin{align}\label{eq:assprimep}
\ass\left(\Ide[\Con(X)\cap \Con(Y)]+I_p\right) \supset \ass\left(\overline{\Ide}[\kappa_X^{-1}(Y)]+I_p\right).
\end{align} We now claim that requiring such a containment for each $p$ in $Y'$, is equivalent to the containment in \eqref{eq:lemmaOpenAssPrime}. 

 Note that, since $p$ is not in $A$ by assumption, we have that $\kappa_A^{-1}(p)$ is not contained in any primary components supported on $\kappa^{-1}(A)$ and hence, when summing with $I_p$, we can replace $\Ide[\kappa_X^{-1}(Y-A)]$ on the left side of \eqref{eq:lemmaOpenAssPrime} by $\Ide[\kappa_X^{-1}(Y)]$. Similarly, the integral closure on the right side may as well be replaced by $\overline{\Ide}[\kappa_X^{-1}(Y)]$ when summing with $I_p$.

 Now suppose \eqref{eq:assprimep} holds for every $p\in Y'$. Recall that $\Ide[\kappa_X^{-1}(Y-A)]$ is the ideal sheaf of the scheme $\overline{\kappa_X^{-1}(Y-A)}$; thus, if we have containment \eqref{eq:assprimep} for every $p\in Y'$, this verifies that the desired condition holds everywhere in $\kappa_X^{-1}(Y-A)$, and thus we automatically have the corresponding containment for the Zariski closures, hence 
we obtain \eqref{eq:lemmaOpenAssPrime}. 

Conversely, suppose that \eqref{eq:lemmaOpenAssPrime} holds and consider any $p\notin A$. All associated primes of $\overline{\Ide}[\kappa_X^{-1}(Y-A)]+I_p$ arise from summing $I_p$ with associated primes of $\overline{\Ide}[\kappa_X^{-1}(Y-A)]$. Since all these primes are also contained in the left hand set, and since we have
\[
\Ide[\kappa_X^{-1}(Y-A)] +I_p=\Ide[\kappa_X^{-1}(Y)]+I_p=I_{\Con(X)}+I_p
\] we obtain the desired containment \eqref{eq:assprimep}. 
\end{proof}

We recall for the reader's convenience that the integral closure $\overline{I}$ of an ideal $I$ in a commutative ring $R$ is defined as follows. Let $I^i$ denote the $i$-th power of $I$, i.e., the set of all $a \in R$ which can be expressed as the product of $i$ elements of $I$. Now $\overline{I}$ consists of all $r$ in $R$ for which there exists some integer $k > 0$ and elements $a_i\in I^i$ satisfying
\[
r^k + a_{k-1}\cdot r^{k-1} + \cdots + a_1 \cdot r + a_0=0.
\] 
In general, if we are only given access to a set of generating elements for $I$, then various algorithmic operations involving $\overline{I}$ become computationally prohibitive even in the relatively benign case $R = \CC[x_0,\ldots,x_n]$. These hard tasks include, for instance, extracting a list of defining polynomials for $\overline{I}$ and testing whether a given $r \in R$ lies inside $\overline{I}$ (see \cite[Sections 6.6 and 6.7]{vasconcelos2004computational} or  \cite[Section 9.3]{vasconcelos2006integral}). Thus, its considerable aesthetic appeal notwithstanding, Proposition \ref{prop:conormCond} does not furnish an efficient algorithmic mechanism for verifying Condition (B). We are therefore compelled to employ a Corollary of this Proposition, where the integral closure has been replaced by a far more tractable object.
	
\section{An Effective Version of Condition (B)}\label{sec:saturations}	

In this section we prove our main result. Let $R$ be the coordinate ring of $\pp^n \times (\pp^n)^*$, and consider a projective variety $X \subset \pp^n$ with conormal map $\kappa_X:\Con(X) \surj X$. For any subvariety $Y \subset X$, the ideal sheaf $\Ide[\kappa_X^{-1}(Y)]$ constitutes an honest homogeneous ideal $I_Y + I_{\Con(X)}$ of $R$, and it therefore makes sense from an algorithmic perspective to perform a primary decomposition \cite[Chapter 4.8]{CLO} of this ideal. Given such a homogeneous ideal $I \lhd R$, we will denote the associated projective subvariety of $\pp^n \times (\pp^n)^*$ by $\bV(I)$. 

\begin{theorem}	Let $X\subset \pp^n$ be a pure dimensional projective variety and $Y$ a nonempty irreducible subvariety of its singular locus $X_{\rm sing}$. Consider a primary decomposition (of $R$-ideals)
\[
{\Ide}[{\kappa_X^{-1}(Y)}]= \bigcap_{i=1}^s Q_i,
\]
and let $\sigma \subset \set{1,2,\ldots,s}$ be the set of indices $i$ for which $\dim \kappa_X(\bV(Q_i)) < \dim Y$. Define 
\begin{align}\label{eq:W}
A := \left[\bigcup_{i \in \sigma}\kappa_X(\bV(Q_i))\right] \cup Y_{\rm sing}.
\end{align} 
Then the pair $(X_\text{\rm reg},Y-A)$ satisfies Condition (B). \label{thm:WhitB_Primary_Decomp}
\label{thm:whitneyconormal}
\end{theorem}
\begin{proof}
    By the remarks following \cite[Thm 3.12]{MumfordOda}, we have 
    \begin{align}\label{eq:primary}
    {\Ide}[\kappa_X^{-1}(Y-A)]=\bigcap_{i \in \rho_A} Q_i,
    \end{align}
    where $\rho_A \subset \set{1,\ldots,s}$ is the collection of all $i$ for which $\kappa_X(\bV(Q_i)) - A$ is nonempty. Since $Y$ is irreducible and $\kappa_X(Q_i)$ is a subvariety of $Y$ for each $i$, we have $i \in \rho_A$ if and only if $\kappa_X(\bV(Q_i))= Y$. Therefore, $\rho_A$ is the complement of $\sigma$ in $\set{1,\ldots,s}$ (note that $Y_{\rm sing}$ is a proper subvariety of Y so $\kappa_X(\bV(Q_i)) - Y_{\rm sing}$ is nonempty for any $i\in \rho_A$ and the remaining factors are those where $i\in \sigma$). By Whitney's result on the existence of stratifications \cite[Theorem 19.2]{whitney1965tangents}, there is a proper subvariety $B \subsetneq Y$ containing $Y_\text{sing}$ so that $(X_\text{reg},Y-B)$ satisfies Condition (B). Thus, Proposition \ref{prop:conormCond} guarantees the containment
	\begin{align}\label{eq:intclos} 
    \ass(\Ide[\Con(X)\cap \Con(Y)] +\Ide[ \kappa_X^{-1}(Y-B)])\supset \ass(\overline{\Ide}[\kappa_X^{-1}(Y-B)]).
		\end{align} 
	Applying \eqref{eq:primary} with $A$ replaced by $B$, we note that $\rho_B$ must contain $\rho_A$ because we have $\kappa_X(\bV(Q_i))=Y$ for every $i \in \rho_A$, and $Y - B$ is nonempty for $B \subsetneq Y$. Setting 
	\[
	C := B \cup \left[\bigcup_{i\in \rho_B-\rho_A}\kappa_X(\bV(Q_i))\right],
	\] we note that Condition (B) automatically holds for $(X_\text{reg},Y-C)$ since $B$ is contained in $C$, whence by  Proposition \ref{prop:conormCond} we have 
	\begin{align}\label{eq:Cwhit}
   \ass( \Ide[\Con(X)\cap \Con(Y)]+\Ide[ \kappa_X^{-1}(Y-C)] )\supset \ass(\overline{\Ide}[\kappa_X^{-1}(Y-C)]).
	\end{align} The argument which gave us $\rho_A \subseteq \rho_B$ also yields $\rho_A \subseteq \rho_C$. We now claim that the opposite containment also holds, whence $\rho_A = \rho_C$. To see this, note that $C$ contains $\kappa_X(\bV(Q_i))$ for all the $i \in \rho_B-\rho_A$, i.e., for all the $i$ satisfying $\kappa_X(\bV(Q_i))\subsetneq Y$. Thus, if $i\in \rho_C$ then $\kappa_X(\bV(Q_i))= Y$ and hence $\rho_C\subseteq \rho_A$, which establishes the claim.  Hence by construction, ${\Ide}[\kappa_X^{-1}(Y-C)]$ and ${\Ide}[\kappa_X^{-1}(Y-A)]$ are identical since both equal $\bigcap_{i \in \rho_A} Q_i$. 
	Using \eqref{eq:Cwhit}, one final appeal to  Proposition \ref{prop:conormCond} confirms that the pair $(X_{\rm reg},Y-A)$ satisfies Condition (B), as desired. 
\end{proof}

We note that the argument given above adapts readily to the scenario where $Y$ is pure-dimensional but not necessarily irreducible. In this case, one defines $\rho_A$ to consist precisely of those $i \in \set{1,\ldots,s}$ for which $\kappa_X(\bV(Q_i))$ equals an irreducible component of $Y$, all of which are equidimensional. Similarly any valid choice of $B$ must be a proper subvariety of some irreducible component of $Y$ and hence satisfy $\dim(B)<\dim(Y)$. The proof then proceeds identically.  

\begin{example}
Consider the projectivized Whitney Cusp $\widetilde{X} \subset \PP^3$ given by \[\widetilde{X}=\bV\left(x_{0}^{2}x_{2}^{2}+x_{0}x_{3}^{3}-x_{1}^{2}x_{3}^{2}\right).\] The singular locus $
\widetilde{X}_{\rm Sing}$ equals $\bV(x_0,x_3)\cup\bV(x_2,x_3), 
$ and the usual (affine) Whitney cusp $X\subset \CC^3$ is obtained by setting $x_0 = 1$. The singular locus of $X$ is the line $Y:=\bV(x_2,x_3)$; note we slightly abuse notation and write $Y$ for both the affine and projective variety defined by $x_2=x_3=0$. Now, we have a primary decomposition
\begin{align*}
	{\Ide}[\kappa_X^{-1}(Y)]=\bigcap_{i=1}^9 Q_i, \text{ where:}
\end{align*}
\[
\xymatrixcolsep{.1in}
\xymatrixrowsep{-.08in}
\xymatrix{ \kappa_X(\bV(Q_1))=\kappa_X(\bV({Q}_2))=Y, &   \kappa_X(\bV({Q}_3))=\bV(x_1,x_2,x_3), \\
\kappa_X(\bV({Q}_4))=\bV(x_0,x_2,x_3), &   \kappa_X(\bV(Q_j))=\emptyset \text{ for }5 \leq j \leq 9.}
\] 
Thus, applying Theorem \ref{thm:WhitB_Primary_Decomp} and setting $x_0=1$ establishes Condition (B) for the pair $X_{\rm reg}=X-Y$ and $Y'=Y-\bV(x_1,x_2,x_3)$, which in turn gives the following strata: $X_{\rm reg}$, $Y'$, and $\bV(x_1,x_2,x_3)$. 
\end{example}

\section{Stratifying Projective Varieties}\label{sec:WSalgorithm} 
	
	In this section we describe a recursive algorithm which uses Theorem \ref{thm:whitneyconormal} to compute Whitney stratifications of pure-dimensional projective varieties. Since each irreducible component of an arbitrary projective variety can be stratified separately, this is by no means a severe restriction. And since various intermediate varieties which get constructed in our algorithm won't satisfy this purity criterion, it will be convenient to let $\Pure_d(Z)$ denote the set of all the pure $d$-dimensional irreducible components of a given projective variety $Z$. These components can be algorithmically extracted by computing the {\em minimal associated primes} of the defining ideal $I_Z$, and the main cost is a Gr\"obner basis computation --- see \cite{decker1999primary}.
	
	\subsection{Computing Conormal Ideals} \label{ssec:computeCon}
	In light of Theorem \ref{thm:whitneyconormal}, we will  often be required to compute (the equations which define) the conormal space $\Con(Z)$ of a given projective variety $Z \subset \pp^n$. The ideal $I_{\Con(Z)} \lhd \CC[x,\xi]$ is extracted in practice as follows. First, we let $I_Z = \langle f_1, \ldots, f_r\rangle$ be any defining ideal of $Z$, and let $c$ be its codimension $n - \dim(Z)$. One computes the Jacobian ideal $\textbf{Jac}_Z$ of $Z$ which is generated by all the $c\times c$ minors of the matrix of partial derivatives
	\[  
	\begin{bmatrix}
	\nicefrac{\partial f_1}{\partial x_0} &\cdots& \nicefrac{\partial f_1}{\partial x_n}\\
	\vdots& \ddots & \vdots \\
	\nicefrac{\partial f_r}{\partial x_0} &\cdots& \nicefrac{\partial f_r}{\partial x_n}\\
	\end{bmatrix},
	\] see \cite[pg.~27--28]{michalek2021invitation}. The singular set $Z_\text{sing} \subset Z$ is (by definition) the zero locus of the Jacobian ideal, namely:
	\[
	Z_\text{sing} = \bV(I_Z+\textbf{Jac}_Z).
	\]
	Now if we let $\textbf{Jac}^\xi_Z$ be the ideal generated by the $(c+1)\times (c+1)$ minors of the $\xi_i$-augmented Jacobian matrix:
	\[
	\begin{bmatrix}
	\xi_0 & \cdots &\xi_n\\
	\nicefrac{\partial f_1}{\partial x_0} &\cdots& \nicefrac{\partial f_1}{\partial x_n}\\
	\vdots& \ddots & \vdots \\
	\nicefrac{\partial f_r}{\partial x_0} &\cdots& \nicefrac{\partial f_r}{\partial x_n}\\
	\end{bmatrix},
	\] then $
	I_{\Con(Z)}$ is given by the saturation $(I_Z+\textbf{Jac}^\xi_Z):(\textbf{Jac}_Z)^\infty$, see  \cite[Eq.~(5.1)]{draisma2016euclidean} or \cite[\S2]{holme1988geometric} for instance. Algorithms for computing ideal sums and saturations are standard fare across computational algebraic geometry, and may be found in \cite{CLO} for instance. 
	
	\subsection{The Decomposition Subroutine} \label{ssec:decompose}
	
	For any pair of projective varieties $X,Y \subset \pp^n$ with $Y \subset X_\text{sing}$, the following subroutine, called {\bf Decompose}, implements the primary decomposition based constructions of Theorem \ref{thm:whitneyconormal}. Although we present the inputs and outputs of all our algorithms as projective varieties, in practice these must be represented on a machine by some choice of generating polynomials of their defining ideals. 
	
	\medskip

	\begin{center}
		\begin{tabular}{|r|l|}
			\hline
			~ & {\bf Decompose}$(Y,X )$ \\
			\hline
			~&{\bf Input:} Projective varieties $Y \subset X$ in $\pp^n$, with $d:=\dim Y$.\\
			~&{\bf Output:} A list of subvarieties $Y_\bullet$ of $Y$. \\
			\hline
			1 & {\bf Set} $Y_\bullet := (Y_d,Y_{d-1},\ldots,Y_0) := (Y,\varnothing,\ldots,\varnothing)$ \\
			2 &  {\bf Set} $J:=I_{\Con(X)}+I_{Y}$ \\
			3 & {\bf For each} primary component $Q$ of a primary decomposition of $J$\\
			4 & \spc {\bf Set} $K := Q\cap \CC[x]$ \\
			5 & \spc {\bf If}  $\dim \bV(K) < \dim Y$   \\
			6 & \spc \spc{\bf Add} $\bV(K)$ to $Y_{ \geq\dim \bV(K)}$\\
			7 & {\bf Return} $Y_\bullet$ \\
			\hline
		\end{tabular} 
	\end{center}
	
	\medskip
	
	The notation in Line 6 is meant to indicate that $V$ is added to $Y_i$ for all $i \geq \dim V$. This subroutine terminates because the {\bf For} loop in line 3 is indexed over primary components of a polynomial ideal, of which there can only be finitely many. The following result is a direct consequence of Theorem \ref{thm:whitneyconormal}.
	
	\begin{proposition}\label{prop:decomp}
		Let $Y \subset X$ be a pair of projective varieties in $\pp^n$ so that $Y \subset X_\text{\rm sing}$ and $\dim Y = d$. If {\bf Decompose} is called with input $(Y,X)$, then:
		\begin{enumerate}
			\item for all $i \in \set{0,\ldots, d-1}$, its output varieties satisfy $Y_i \subset Y_{i+1}$; also, 
			\item $Y_{d-1}$ is a (possibly empty) subvariety of $Y$ with $\dim Y_{d-1} < \dim Y$; and finally,
			\item all points of $Y_{\rm reg}$ where Condition (B) fails with respect to $X_{\rm reg}$ lie in $Y_{d-1}$, so the pair $(X_\text{\rm reg},Y_\text{\rm reg}-Y_{d-1})$ satisfies Condition (B). 
		\end{enumerate}
	\end{proposition}
	\begin{proof} 
		The first assertion holds because of Line 6: whenever a $Z:=\bV(K)$ is added to $Y_{\dim Z}$, it is also added to all the subsequent $Y_i$ with $i > \dim Z$.  
		For the second assertion, from Line 2, $J$ is the ideal of the scheme $\kappa_X^{-1}(Y)$ and hence the image of any primary component under $\kappa_X$ must be contained in $Y$, but $Y$ is irreducible and $\dim(\bV(K))<\dim(Y)$, hence all $\bV(K)$ are proper subvarieties of $Y$. The third assertion follows directly from Theorem \ref{thm:whitneyconormal}.
\end{proof}
	
	The output $Y_\bullet$ of {\bf Decompose} described above need not constitute a Whitney stratification of the input variety $Y$ --- Proposition \ref{prop:decomp} does not guarantee that Condition (B) holds among successive differences of the form $Y_i - Y_{i-1}$. 
	
	\subsection{The Main Algorithm} \label{ssec:mainalg1}
	
	Let $X\subset \pp^n$ be a pure $k$-dimensional complex projective variety defined by a radical homogeneous ideal $I_X$. The algorithm {\bf WhitStrat}, described below, takes in $X$ as input and returns a nested sequence of its subvarieties $X_\bullet$. An essential part of this algorithm is a {\bf Merge} statement, which accepts two nested sequences of varieties $V_\bullet$ and $W_\bullet$ of lengths $k$ and $d$, with $k > d$. To {\bf Merge} $V_\bullet$ with $W_\bullet$, one updates the longer sequence $V_\bullet$ via the following rule:
	\[
	V_i \gets V_i \cup W_{\max(i,d)}.
	\]
	The property $V_i \subset V_{i+1}$ continues to hold after such an operation. The following algorithm {\bf WhitStrat} uses {\bf Decompose} in order to construct Whitney stratifications of pure-dimensional projective varieties.
	
	\medskip
	
	\begin{center}
		\begin{tabular}{|r|l|}
			\hline
			~ & {\bf WhitStrat}$(X)$ \\
			\hline
			~&{\bf Input:} A pure $k$-dimensional variety $X\subset \pp^n$.\\
			~&{\bf Output:} A list of subvarieties $X_\bullet$ of $X$. \\
			\hline
			1 & {\bf Set} $X_\bullet := (X_k,X_{k-1},\ldots,X_0) := (X,\varnothing,\ldots,\varnothing)$ \\
			2 & {\bf Compute} $X_{\rm Sing}$ and $\mu := \dim(X_{\rm Sing})$ \\
			3 & {\bf For each} irreducible component $Z$ of $X_\text{sing}$  \\
			4 & \spc {\bf Add} $Z$ to $X_{\geq \dim Z}$ \\
			5 & {\bf For each} $d$ in $(\mu, \mu-1,\ldots,1,0)$ \\
			6 & \spc {\bf Merge} $X_\bullet$ with ${\bf Decompose}(\Pure_d(X_d),X)$ \\
			7 & \spc {\bf Merge} $X_\bullet$ with ${\bf WhitStrat}(\Pure_d({X_d}))$\\
			8 & {\bf Return} $X_\bullet$ \\
			\hline
		\end{tabular}
	\end{center}
	
	\medskip
	
	To verify that this algorithm terminates, we note that the {\bf For} loop on Line 3 runs once per irreducible component of $X_\text{sing}$, of which there can only be finitely many. And the {\bf For} loop on Line 5 will terminate provided that the recursive call on Line 7 terminates; but the dimension of the variety $X_d$ is bounded above by $\mu < k$, i.e., it is strictly less than the dimension of the input variety. Thus, the recursion terminates after finitely many steps and produces a nested sequence $X_\bullet$ of subvarieties of $X$. Our goal in the next section is to establish that $X_\bullet$ constitutes a valid Whitney stratification of $X$.
	
	\subsection{Correctness}
	
	Let $X_\bullet(d)$ denote the nested sequence of varieties $X_\bullet$ as they stand at the {\em end} of the $d$-th iteration of {\bf For} loop (in Line 5 of {\bf WhitStrat}) where $d$ ranges over $(\mu, \mu-1, \ldots, 1,0)$. These $X_i(d)$ fit into a $(\mu+1) \times (k+1)$ grid of projective varieties and inclusion maps:
	\begin{align}\label{eq:grid}
	\xymatrixrowsep{.2in}
	\xymatrixcolsep{.33in}
	\begin{gathered}
	\xymatrix{
		X_k(\mu) \ar@{_{(}->}[d] & X_{k-1}(\mu) \ar@{_{(}->}[d] \ar@{_{(}->}[l] \ar@{_{(}->}[d] & \cdots \ar@{_{(}->}[l]  & X_{1}(\mu) \ar@{_{(}->}[d] \ar@{_{(}->}[l]& X_0(\mu) \ar@{_{(}->}[d] \ar@{_{(}->}[l] \\
		X_{k}(\mu-1) \ar@{_{(}->}[d] & X_{k-1}(\mu-1) \ar@{_{(}->}[l] \ar@{_{(}->}[d] & \cdots \ar@{_{(}->}[l] & X_{1}(\mu-1) \ar@{_{(}->}[d] \ar@{_{(}->}[l] & X_0(\mu-1) \ar@{_{(}->}[l] \ar@{_{(}->}[d] \\
		\myvdots \ar@{_{(}->}[d] & \myvdots \ar@{_{(}->}[d] & \ddots & \myvdots \ar@{_{(}->}[d] & \myvdots \ar@{_{(}->}[d] \\
		X_k(1) \ar@{_{(}->}[d] & X_{k-1}(1) \ar@{_{(}->}[d] \ar@{_{(}->}[l] & \cdots \ar@{_{(}->}[l] & X_{1}(1) \ar@{_{(}->}[l] \ar@{_{(}->}[d] & X_0(1) \ar@{_{(}->}[l] \ar@{_{(}->}[d] \ar@{_{(}->}[l] \\
		X_k(0) & X_{k-1}(0) \ar@{_{(}->}[l] & \cdots \ar@{_{(}->}[l] & X_{1}(0) \ar@{_{(}->}[l] & X_0(0) \ar@{_{(}->}[l]}
	\end{gathered}
	\end{align}
	The vertical inclusion maps arise since each $X_\bullet(d)$ is obtained from the previous $X_\bullet(d+1)$ by performing two {\bf Merge} operations (in Lines 6 and 7 respectively). Note that each iteration of the {\bf For} loop in Line 5 moves us from one row to the next, until at last the bottom row (corresponding to index $d = 0$) contains the output. 
	
	\begin{remark}\label{rem:gridcols} Before entering the {\bf For} loop of Line 5, the sequence $X_\bullet$ contains the input variety $X$ in the top dimension (i.e., $X_k = X$) and irreducible components of $X_\text{sing}$ in lower dimensions. It follows that the columns of our grid \eqref{eq:grid} come in three flavours, depending on the index $i$:
		\begin{enumerate}
			\item the left-most column (with index $i=k$) identically equals $X$, i.e., $X_k(d) = X$ regardless of the row index $d$ in $\set{0,\ldots,k}$;
			\item the next few columns (with index $k > i \geq \mu$) are also constant --- independent of the row index $d$, each $X_i(d)$ in this range equals $X_\text{sing}$; and finally,
			\item the $i$-th column for $\mu > i \geq 0$ stabilizes below its $(i+1)$-indexed entry:
			\begin{align}\label{eq:triangular}
			X_i(d) = X_i(i+1) \quad \text{ for all } \quad 0 \leq d \leq i < \mu.
			\end{align}
		\end{enumerate}
		These three assertions follow from the observation that during the $d$-indexed iteration of the {\bf For} loop in Line 5, the two {\bf Merge} operations (from Lines 6 and 7) which produce the row $X_\bullet(d)$ are only allowed to merge subvarieties of $X_d(d+1)$ to the preceding row $X_\bullet(d+1)$. 
	\end{remark}

	Define the successive differences across the rows of the grid \eqref{eq:grid}, i.e.,
	\begin{align}\label{eq:sdiff}
	S_i(d) := X_i(d) - X_{i-1}(d).
	\end{align}
	It follows from Lines 2-4 of {\bf WhitStrat} that we have the containment 
	\[
	X_i(d)_\text{Sing} \subset X_{i-1}(d) \text{ for all } d \leq i,
	\] whence $S_i(d)$ is a smooth manifold whenever $d \leq i$. Therefore, $S_i(0)$ is smooth for all $i$. In the remainder of this Section, we will now show that $S_i(0)$ constitutes the $i$-strata of a Whitney stratification $X_\bullet(0)$ of the input variety $X$. Here is a first step in this direction.
	
	\begin{proposition}\label{prop:topB}
		The pair $(S_k(i),S_{d}(i))$ satisfies Condition (B) for all $0 \leq i\leq d \leq k$. In other words, $X_{d-1}(i)$ contains all points of $X_d(i)_\text{\rm reg}$ where Condition (B) fails with respect to $S_k(i)$. 
	\end{proposition}
	\begin{proof} By Remark \ref{rem:gridcols}, we may safely restrict to the case $0 \leq i\leq d \leq \mu$ since we have 
		\[
		X_\mu(d)=X_{\mu+1}(d)=\cdots =X_{k-1}(d)
		\] for all $d$ by Remark \ref{rem:gridcols}(2). Let $V_\bullet$ denote the output of {\bf Decompose} obtained (in Line 6) during the $d$-th iteration of the {\bf For} loop in Line 5. It follows from Proposition \ref{prop:decomp}(3) that the pair
		\[
		P := \left(X_\text{reg},X_d(d+1)_\text{reg}-V_d\right)
		\] satisfies Condition (B). By Proposition \ref{prop:decomp}(2), we know that $V_d$ is a subvariety of $X_d(d+1)$ of dimension strictly smaller than $d$. Since $V_\bullet$ is merged with $X_\bullet(d+1)$ in Line 6 en route to producing $X_\bullet(d)$, we know that $V_d$ must in fact be a subvariety of $X_{d-1}(d)$. Now note that 
		\begin{align*}
		S_k(d) &= X_k(d) - X_{k-1}(d) & \text{ by \eqref{eq:sdiff}}, \\
		&= X - X_{k-1}(d) & \text{ by Remark \ref{rem:gridcols}(1)}. 
		\end{align*} Since $\dim X_{k-1}(d) < k$, we know that $S_k(d)$ is dense in $X_\text{reg}$. Moreover, since $V_d$ is entirely contained in $X_{d-1}(d)$, we may apply Proposition \ref{prop:condBdense} to the pair $P$ above and conclude that the new pair 
		\[
		P' := \left(S_k(d),X_d(d+1)_\text{reg}-X_{d-1}(d)\right)
		\]
		also satisfies Condition (B). Now let $W_\bullet$ be the output of the recursive call to {\bf WhitStrat} in Line 7 during the $d$-th iteration of the {\bf For} loop in Line 5. From Lines 2-4, we deduce that $X_d(d+1)_\text{sing} \subset W_{d-1}$, and after the {\bf Merge} operation of Line 7 we are also guaranteed $W_{d-1} \subset X_{d-1}(d)$. Putting these containments together gives
		\[
		X_d(d+1)_\text{sing} \subset W_{d-1} \subset X_{d-1}(d). 
		\]
		By \eqref{eq:triangular}, we have $X_d(d+1) = X_d(d)$, whence $X_d(d)_\text{sing} \subset X_{d-1}(d)$. Therefore, the difference $S_d(d) := X_d(d) - X_{d-1}(d)$ in fact equals $X_d(d+1)_\text{reg}-X_{d-1}(d)$. Using this in the pair $P'$ guarantees that the pair $(S_k(d),S_{d}(d))$ satisfies Condition (B). Finally, the conclusion for the pair $(S_k(i),S_{d}(i))$ follows since subsequent iterations of the the {\bf For} loop (corresponding to lower $d$ values) do not alter $X_{\geq d}$.
	\end{proof}
	
	Next, we will establish that Condition (B) is satisfied by arbitrary pairs of successive differences in \eqref{eq:grid} for sufficiently small row index.
	
	\begin{proposition}\label{prop:allB}
		The pair $(S_j(i),S_i(i))$ satisfies Condition (B) for all $0 \leq i < j \leq k$.
	\end{proposition}
	\begin{proof} From Remark \ref{rem:gridcols}, we have $S_j(i)=S_j(j)$ since $i < j$. Let $W_\bullet$ be the output of the recursive call in Line 7 of {\bf WhitStrat} during the $d = j$ iteration of the {\bf For} loop in Line 5, so we have $W_j = \Pure_j(X_j(j))$. Recall by \eqref{eq:sdiff} that $S_j(i) = X_j(i) - X_j(i-1)$; now any irreducible component $Y \subset X_j(i)$ with $\dim Y < j$ must also lie in $X_{j-1}(i)$, whence 
		\begin{align*}
		S_j(i) &=\Pure_j(X_j(j))-X_{j-1}(i) \nonumber \\ 
		&=W_j-X_{j-1}(i). 
		\end{align*} Now $W_{j-1}$ lies in $X_{j-1}(j)$ because of the {\bf Merge} operation in Line 7 of {\bf WhitStrat}, and in turn $X_{j-1}(j)$ is a subvariety of $X_{j-1}(i)$ as described in \eqref{eq:grid}. Since both varieties $W_{j-1} \subset X_{j-1}(i)$ have dimension strictly smaller than $j$, we have that $S_j(i)$ is dense in $W_j-W_{j-1}$. Thus, it suffices to show that all points in $W_j \cap X_i(i)$ where Condition (B) fails with respect to $S_j(i)$ lie within $W_j\cap X_{i-1}(i)$. To confirm this, note that by construction $X_i(i)$ is a union of the form $W_i \cup Z$, where the subvariety $Z \subset X_i(i)$ has $\dim(Z) \leq i$. First we consider the case where $Z$ is empty; in this case, we know from Proposition \ref{prop:decomp}(3) that the pair
		\[
		(W_j-W_{j-1},W_i-W_{i-1})
		\] satisfies Condition (B), so the desired result follows immediately from Proposition \ref{prop:condBdense}. On the other hand, if $Z$ is nonempty, we can assume without loss of generality that $Z$ is not contained in $W_i$. Now any point of $W_j \cap (X_i(i)-W_i)= W_j\cap (Z-W_i)$ where Condition (B) fails with respect to $W_j$ must lie in $W_{j-1}$ by Proposition \ref{prop:topB}. Thus, no such point lies in $S_j(i)$, and it remains to show that all points in $W_i\cap X_i(i)=W_i$ where Condition (B) fails with respect to $S_j(i)$ are contained in $X_{i-1}(i)$. But since $W_{i-1}\subset X_{i-1}(i)$ due to the {\bf Merge} operation, this follows immediately from Proposition \ref{prop:topB}.
	\end{proof}
	
	We can now confirm that the output of {\bf WhitStrat} constitutes a valid Whitney stratification of $X$.
	
	\begin{theorem}
		When called on a pure $k$-dimensional complex projective variety $X\subset \pp^n$, the output $X_\bullet$ of {\bf WhitStrat} forms a Whitney stratification of $X$.  \label{thm:WSCorrect}
	\end{theorem}
	\begin{proof} 
		As remarked after \eqref{eq:sdiff}, each $S_i(d)$ is smooth for $d \leq i$. The conclusion follows immediately from Proposition \ref{prop:allB} since for any pair $(S_j(i),S_i(i))$ with $0\leq i \leq j \leq k$ we have $S_i(i)=S_i(0)$ and $S_j(i)=S_j(0)$ by Remark \ref{rem:gridcols}. Thus, Condition (B) holds for every pair $(S_j(0),S_i(0))$ with $i \leq j \leq k$. 
	\end{proof}

	This algorithm can also be used to stratify affine complex varieties via the following dictionary: for each affine complex variety $X \subset \CC^n$, we write $PX \subset \pp^n$ for its {\em projective closure} \cite[Proposition~2.8]{michalek2021invitation}, which is constructed as follows. Let $\set{f_1, \ldots, f_r}$ be a Gr\"obner basis for the defining ideal $I_X \lhd \CC[x_1,\ldots,x_n]$. For each $i$ in $\set{1,\ldots,r}$, write $F_i$ for the homogenisation of $f_i$ in $\CC[x_0,x_1,\ldots,x_n]$. Then, the projective closure $PX \subset \pp^n$ is the complex projective variety given by $\bV(F_1,\ldots,F_r)$. Conversely, one can recover $X$ from $PX$ by {\em dehomogenizing}, i.e., by setting $x_0 = 1$ in each defining polynomial $F_i$.
	
	\begin{corollary}
		Let $X\subset \CC^n$ be a pure $k$-dimensional affine complex variety and let $PX \subset \pp^n$ be its projective closure. If $PX_\bullet$ is the output of  {\bf WhitStrat}$(PX)$, then a valid Whitney stratification of $X$ is given by $X_\bullet$, where each $X_i$ is the dehomogenization of $PX_i$. \label{cor:affine_Whit_Strat}
	\end{corollary}
	\begin{proof}
		Since $PX_\bullet$ defines a Whitney stratification of $PX$, it follows from Proposition \ref{prop:condBdense} that intersecting each $PX_i$ with a dense subset $D \subset PX$ constitutes a Whitney stratification of $D$. The conclusion follows by considering $D = PX - \bV(x_0)$. 
	\end{proof}
	
	It should be noted that the stratifications produced by {\bf WhitStrat} may not be minimal; and moreover, the stratification of $X$ described in the preceding Corollary may not be minimal even if the output stratification of $PX$ is minimal. 
	
	\section{Flag-Subordinate Stratifications}\label{sec:flagstrat}
	
	By a {\em flag} $\bF_\bullet$ on a variety $X$ we mean any finite nested set of subvarieties of the form
	\[
	\varnothing = \bF_{-1}X \subset \bF_0X \subset \bF_1X \subset \cdots \subset \bF_{\ell-1}X \subset \bF_\ell X = X.
	\]
	If $X$ is projective, we implicitly require each $\bF_iX$ to also be projective. We call $\ell$ the {\em length} of the flag $\bF_\bullet$. Aside from these containments, there are no restrictions on the dimensions of the individual $\bF_iX$; and in particular, we do not require successive differences $\bF_iX - \bF_{i-1}X$ to be smooth manifolds, let alone satisfy Condition (B).
	
	\begin{definition}\label{def:flagsubstrat}
		Let $X \subset \pp^n$ be a projective variety and $\bF_\bullet$ a flag on $X$ of length $\ell$. A Whitney stratification $X_\bullet$ of $X$ is {\bf subordinate} to $\bF_\bullet$ if for each stratum $S \subset X$ of $X_\bullet$ there exists some $j = j(S)$ in $\set{0,\ldots,\ell}$ satisfying $S \subset (\bF_jX - \bF_{j-1}X)$.
	\end{definition}
	
	\noindent It is crucial to note that the number $j(S)$ from the preceding Definition need not equal $\dim S$, and that one does not require $j(S) = j(S')$ whenever $\dim S = \dim S'$. 
	
	Fix a pure-dimensional complex projective variety $X \subset \pp^n$ as well as a flag $\bF_\bullet$ on $X$ of length $\ell < \infty$. The following subroutine accepts as input any (not necessarily pure dimensional) subvariety $W \subset X$ along with the flag $\bF_\bullet$, and constructs the {\em induced flag} $\bF'_\bullet$ on $W$ defined by 
	\[
	\bF'_jW := W \cap \bF_jX
	\] for all $j$ in $\set{0,1,\ldots,\ell}$.
	
	\medskip
	
	\begin{center}
		\begin{tabular}{|r|l|}
			\hline
			~ & {\bf InducedFlag}$(W,\bF_\bullet)$ \\
			\hline
			~&{\bf Input:} A subvariety $W \subset X$ and a flag $\bF_\bullet$ on $X$ of length $\ell$.\\
			~&{\bf Output:} A flag $\bF'_\bullet$ on $W$ of length $\ell$. \\
			\hline
			1 & {\bf Set} $\bF'_\bullet W := (\bF'_\ell W, \ldots, \bF'_0W) := (\varnothing,\ldots,\varnothing)$\\
			2 & {\bf For each} irreducible component $V$ of $W$  \\
			3 & \spc {\bf Add} $V$ to $\bF'_{i}W$ {\bf for all} $\bF'_{i}$ where $V\subset \bF'_{i}$\\
			4 & \spc {\bf For each} $j$ with $\dim(\bF_jX \cap V) < \dim V$ \\
			5 & \spc \spc {\bf Add} $V_j := (\bF_jX \cap V)$ to $\bF'_{i} W$ {\bf for all} $\bF'_{i}$ where $V_j\subset \bF'_{i}$ \\
			6 & {\bf Return} $\bF'_\bullet W$\\
			\hline
		\end{tabular}
	\end{center}
	
	\medskip
	
	The strategy for producing an $\bF_\bullet$-subordinate stratification of $X$ is to modify the algorithms of Section \ref{sec:WSalgorithm} as follows: whenever one wishes to {\bf Add} an irreducible $i$-dimensional subvariety $W \subset X$ to the output sequence $X_{\geq i}$, one {\bf Merge}s $X_\bullet$ with {\bf InducedFlag}$(W,\bF_\bullet)$ instead. For completeness, we have written out the modified versions of both {\bf Decompose} and {\bf WhitStrat} (the originals can be found in Sections \ref{ssec:decompose} and \ref{ssec:mainalg1} respectively). 
	
	\subsection{Flag-Subordinate Stratification Algorithms}\label{ssec:flagstrat}
	
	Here is the flag-subordinate avatar of {\bf Decompose}; as promised, it only differs from the original in Line 6: the statement which {\bf Add}ed an irreducible variety to a sequence has now been replaced with a {\bf Merge}.
	
	\medskip
	
\begin{center}
		\begin{tabular}{|r|l|}
			\hline
			~ & {\bf DecomposeFlag}$(Y,X,\bF_\bullet)$ \\
			\hline
			~&{\bf Input:} Proj. varieties $Y \subset X$ with $d:=\dim Y$ and a flag $\bF_\bullet$ on $X$.\\
			~&{\bf Output:} A list of subvarieties $Y_\bullet \subset Y$. \\
			\hline
			1 & {\bf Set} $Y_\bullet := (Y_d,Y_{d-1},\ldots,Y_0) := (\varnothing,\ldots,\varnothing)$ \\
			2 &  {\bf Set} $J:=I_{\Con(X)}+I_{Y}$ \\
			3 & {\bf For each} primary component $Q$ of a primary decomposition of $J$\\
			4 & \spc {\bf Set} $K := Q\cap \CC[x]$ \\
			5 & \spc {\bf If}  $\dim \bV(K) < \dim Y$   \\
			6 & \spc \spc {\bf Merge} $Y_\bullet$ with {\bf InducedFlag}$(\bV(K),\bF_\bullet)$\\
			7 & {\bf Return} $Y_\bullet$ \\
			\hline
		\end{tabular}
	\end{center}
	
	\medskip
	
	And here is the variant of {\bf WhitStrat} which produces an $\bF_\bullet$-subordinate stratification of $X \subset \pp^n$. Again, the only difference occurs in Line 4.
	
	\medskip
	
	\begin{center}
		\begin{tabular}{|r|l|}
			\hline
			~ & {\bf WhitStratFlag}$(X,\bF_\bullet)$ \\
			\hline
			~&{\bf Input:} A pure $k$-dimensional variety $X\subset \pp^n$ and a flag $F_\bullet$ on $X$.\\
			~&{\bf Output:} A list of subvarieties $X_\bullet \subset X$. \\
			\hline
			1 & {\bf Set} $X_\bullet := (X_k,X_{k-1},\ldots,X_0) := (X,\varnothing,\ldots,\varnothing)$ \\
			2 & {\bf Compute} $X_{\rm Sing}$ and $\mu := \dim(X_{\rm Sing})$ \\
			3 & {\bf Set} $X_d = X_{\rm Sing}$ for all $d$ in $\set{\mu,\mu+1,\ldots,k-1}$\\
			4 & {\bf Merge} $X_\bullet$ with {\bf InducedFlag}$(X_\text{sing},\bF_\bullet)$\\
			5 & {\bf For each} $d$ in $(\mu, \mu-1,\ldots,1,0)$ \\
			6 & \spc {\bf Merge} $X_\bullet$ with ${\bf DecomposeFlag}(X_d,X,\bF_\bullet)$ \\
			7 & \spc {\bf Merge} $X_\bullet$ with ${\bf WhitStratFlag}(X_d,\bF_\bullet)$\\
			8 & {\bf Return} $X_\bullet$ \\
			\hline
		\end{tabular}
	\end{center}
	
	\medskip
	
	\subsection{Correctness}
	
	The following result confirms that {\bf WhitStratFlag} produces valid flag-subordinate Whitney stratifications.
	
	\begin{theorem} Let $X\subset \pp^n$ be a pure dimensional complex projective variety and $\bF_\bullet$ a flag on $X$. When called with input $(X,\bF_\bullet)$, the algorithm {\bf WhitStratFlag} terminates and its output $X_\bullet$ is an $\bF_\bullet$-subordinate Whitney stratification of $X$. \label{thm:Stat_wrt_Flag}
	\end{theorem}
	\begin{proof}
		Termination follows for the same reasons as the ones used for {\bf WhitStrat}, and since $X_\bullet$ is a valid Whitney stratification follows from Theorem \ref{thm:WSCorrect}. 
		Thus, it remains to show that each connected component $S$ of $X_i-X_{i-1}$ is contained entirely in a single $\bF_jX-\bF_{j-1}X$. Any such $S$ can be written as $Y - X_{i-1}$, where $Y$ is an irreducible component of $X_i$. Note that, in particular, this $Y$ will appear as a $V$ in Line 3 of the {\bf InducedFlag} subroutine when it is called with first input $X_i$. Let $j$ be the smallest index of the flag $\bF_\bullet$ for which $Y \subset \bF_jX$ holds. Then $\dim(Y \cap \bF_\ell X) < i$ for every $\ell < j$, and so $Y_\ell=Y \cap \bF_\ell X$ is {\bf Add}ed to $X_\bullet$ via Line 5 of the {\bf InducedFlag} subroutine and hence $Y_\ell$ is contained in $X_m$ for some $m < i$. Since $X_m \subset X_i$, it follows that 
		\[
		S \cap (\bF_{p}X - \bF_{p-1}X) = \varnothing \text{ whenever } p<j.
		\] But since $S \subset \bF_{j}X$ and $Y$ is irreducible, we have $Y \subset \bF_{j}X$, whence $Y\subset \bF_{p}X$ for all $p\geq j$. Thus, $S$ also has empty intersections with $\bF_{p}X -\bF_{p-1}X$ for $p > j$, and the desired result follows. 
	\end{proof}
	
	We note in passing that the algorithms described in this Section can also be used to produce flag-subordinate stratifications of affine complex varieties. As before, we will write $PX$ for the projective closure of each affine variety $X$; and given a flag $\bF_\bullet$ on $X$, we write $\text{\bf P}\bF_\bullet$ for the flag on $PX$ defined by 
	\[
	\text{\bf P}\bF_i(PX) := P(\bF_iX).
	\]
	The following result forms a natural flag-subordinate counterpart to Corollary \ref{cor:affine_Whit_Strat}.
	
	\begin{corollary}
		Let $X\subset \CC^n$ be a complex 
		variety and let $\bF_\bullet$ be a flag on $X$. Writing $PX_\bullet$ for the output of {\bf WhitStratFlag} when called with input $(PX,\text{\bf P}\bF_\bullet)$, its dehomogenization $X_\bullet$ constitutes an $\bF_\bullet$-subordinate Whitney stratification of $X_\bullet$. \label{cor:staratWrtFlagAffine}	\end{corollary}
	\begin{proof}
		This follows immediately from Corollary \ref{cor:affine_Whit_Strat} and Theorem 
		\ref{thm:Stat_wrt_Flag} since the containment relations between varieties are preserved by both projective closures and dehomogenizations. 
	\end{proof}
	
	Our motivation for computing flag-subordinate Whitney stratifications stems from the desire to algorithmically stratify algebraic maps between projective varieties.
	
	\section{Stratifying Algebraic Maps} \label{sec:mapstrat}
	
	A continuous map between topological spaces is called {\em proper} if pre-images of compact sets are compact. Reproduced below is the content of \cite[Definition~3.5.1]{brasseletSeadeSuwa}, which highlights a natural class of maps between Whitney stratified spaces; we recall for the reader's convenience that the tangent space at each point $p$ on a smooth manifold $S$ is denoted $T_pS$. 
	
	\begin{definition}\label{def:stratmap3}
		Let $X$ and $Y$ be Whitney stratified spaces. A proper map $f:X \to Y$ is called a {\bf stratified map} if for each stratum $S \subset X$ there exists a a stratum $R \subset Y$ so that 
		\begin{enumerate}
			\item the image $f(S)$ is wholly contained in $R$; and moreover,
			\item at each point $x$ in $S$, the Jacobian $J_x(f|_S):T_xS \to T_{f(x)}R$ is a surjection.
		\end{enumerate}
	\end{definition}
	
	\noindent We refer to any pair $(X_\bullet,Y_\bullet)$ of Whitney stratifications of $X$ and $Y$ which satisfy the above requirements as a stratification of $f$.  It follows from {\em Thom's first isotopy lemma} \cite[Proposition 11.1]{Mather2012} that if $f:X \to Y$ is a stratified map in the sense of this definition, then for each stratum $R \subset Y$ the restriction
	\[
	f|_{f^{-1}(R)}:f^{-1}(R) \to R
	\]
	has the structure of a fibration (with possibly singular fibers). 
	
	Our aim here is to algorithmically construct stratifications tailored to algebraic maps between projective varieties. We describe these maps in terms of the coordinate ring $\CC[x] := \CC[x_0,\ldots,x_n]$ of $\pp^n$.
	
	\begin{definition}\label{def:projmap}
		A {\bf projective morphism} $f:X \to \pp^m$ consists of an $(m+1)$-tuple of homogeneous polynomials $f_i$ in $\CC[x]$, i.e.,
		\[
		f(x) = \left(f_0(x), \dots, f_m(x)\right),
		\] 
		where $d := \deg(f_i)$ is constant for all $i$ and where $ X\cap \bV(f_0, \dots, f_m)=\emptyset$.
	\end{definition} 
	
	\noindent Projective morphisms as defined above are always proper \cite[Ch II, Thm 4.9]{hartshorne}, so at least that requirement of Definition \ref{def:stratmap3} holds automatically. We will restrict to the case where $X$ is pure dimensional; and for each projective morphism $f:X \to \pp^m$ as defined above, we can always consider some pure-dimensional projective variety $Y \subset \pp^m$ which contains the image $f(X)$, whence $f$ constitutes an algebraic map $X \to Y$. We now seek to describe an algorithm which will produce a stratification $(X_\bullet,Y_\bullet)$ for any generic triple $(X,Y,f)$. Both the genericity condition and the algorithm itself make essential use of the Thom-Boardman flag of $f$, which is described below.
	
	\subsection{The Thom-Boardman Flag}\label{ssec:TBflag} 
	Let $f:X\to \pp^m$ be a projective morphism and consider its Jacobian operator 
	\[
	Jf := 
	\begin{bmatrix}
	\nicefrac{\partial f_0}{\partial x_0} &\cdots& \nicefrac{\partial f_0}{\partial x_n}\\
	\vdots& \ddots & \vdots \\
	\nicefrac{\partial f_m}{\partial x_0} &\cdots& \nicefrac{\partial f_m}{\partial x_n}\\
	\end{bmatrix}.
	\]
	This matrix of polynomials is not well-defined on $\pp^n$ in the sense that its evaluation $J_xf$ need not equal $J_yf$ for a pair of projectively equivalent points $x \sim y$ in $\CC^{n+1}$. However, we always have $\text{rank }J_xf = \text{rank }J_yf$ in this case. For each $p \in \pp^n$, we will denote by $\text{rank }J_pf$ the rank of $J_xf$ for any (necessarily nonzero) $x \in \CC^{n+1}$ which maps to $p$ under the canonical surjection $(\CC^{n+1}-\set{0}) \surj \pp^n$.
	
	Let {$k \leq  \min(n+1,m+1)$} be the largest rank attained by $J_pf$ as $p$ ranges over $\pp^n$ and consider, for each $i$ in $\set{0,\ldots,k+1}$, the homogeneous ideal $\text{\bf Jac}_i f$ of $\CC[x]$ generated by all $i \times i$ minors of $Jf$. It follows by cofactor expansion that these ideals fit into a descending sequence, and so the corresponding projective varieties $\bT^+_i\pp^n := \bV(\text{\bf Jac}_i f)$ form a flag of length $k+1$ on $\pp^n$. By construction, we have 
	\begin{align} \label{eq:tbrank}
	\bT^+_i\pp^n = \set{p \in \pp^n \mid \text{rank }J_pf \leq i-1} 
	\end{align}
	for $0 \leq i \leq k+1$. We call $\bT^+_\bullet$ the  the {Thom-Boardman flag} \cite{boardman,thom} of $f$ on $\pp^n$. 	Our main focus here is not on the flag $\bT^+_\bullet$, but rather on the flag induced by $\bT^+_\bullet$ on the domain variety $X \subset \pp^n$. The desired genericity condition on $f$ is described below. 
	
	\begin{definition}\label{def:generic}
		The projective morphism $f:X \to \pp^m$ is {\bf generic} if the varieties $X$ and $\bT_i^+ \PP^n$ intersect transversely in $\pp^n$ for each $i$ in $\set{0, \ldots, k+1}$.
	\end{definition}
	
	Although we will not require this fact here, it is known that for a dense subset of algebraic morphisms, each $\bT^+_i\pp^n$ is a subvariety of $\pp^n$ of dimension
	\[
	\dim \bT^+_{k'+1-i}\pp^n = (n+1) - i \cdot (i+|m-n|),
	\] with $k' := \min(n+1,m+1)$. A complete derivation of this dimension formula along with other properties of Thom-Boardman singularities can be found in \cite[Chapter VI, Part I, \S1]{golubitsky}. 
	
	\begin{definition}\label{def:Tflag}
		The {\bf Thom-Boardman flag of $f$ on $X$}, denoted $\bT_\bullet$, is defined for all $0 \leq i \leq k+1$ via the intersection ${\bT_i}X := \bT^+_i\pp^n \cap X$. Equivalently, $\bT_iX := \bV(I_X + \text{\bf Jac}_i f)$, where $I_X \lhd \CC[x]$ is the defining ideal of $X$.
	\end{definition}
	
	Applying $f$ to the $\bT_\bullet$ produces a flag on $f(X)$ which extends to a flag on $Y$.
	
	\begin{definition}\label{def:Bflag}
		Given the Thom-Boardman flag $\bT_\bullet$ on $X$, its {\bf image} is the flag $\text{\bf B}_\bullet$ of length $k+2$ on $Y$ defined by setting
		\[
		\bB_iY := \begin{cases}
		f(\bT_iX) & i \leq k+1, \\
		Y & i=k+2.
		\end{cases}
		\]
		(Note that by assumption $\bB_{k+1}Y := f(X)$ is a subvariety of $\bB_{k+2}Y := Y$.) 
	\end{definition}
	
	\subsection{Computing Images and Pre-Images} To compute the flag $\bB_\bullet$ in practice, we require a mechanism for producing equations for $f(X')$ where $X' \subset X$ is a subvariety of $X$ --- note that this image $f(X')$ is Zariski closed in $\pp^m$ because $X$ is projective (see \cite[Theorem~4.22]{michalek2021invitation}). This can be accomplished using elimination. To this end, consider the ideal 
	\[
	J := \ip{y_0-uf_0(x), \dots, y_m-uf_m(x)}
	\] in the ring $\CC[u,x_0, \dots, x_n,y_0,\ldots,y_m]$ and set $J_\Gamma := J\cap \CC[x,y]$. Let $\Gamma(X')\subset \pp^n \times \pp^m$ be the {graph} of the restricted map $f|_{X'}$, which is defined by the bi-homogeneous ideal 
	\[
	I_{\Gamma(X')} := I_{X'}+J_\Gamma.
	\]
	Now $f(X')\subset \pp^m$ is given by the elimination ideal 
	\[
	I_{f(X')} := I_{\Gamma(X')} \cap \CC[y].
	\]
	
	We will also require the dual operation to implement our algorithm; i.e., given some subvariety $Y' \subset Y$, we wish to algebraically compute its pre-image $f^{-1}(Y')$ within $X$. Let $\gamma$ be the natural map $X \to \pp^n \times \pp^m$ sending each $x$ to the pair $(x,f(x))$, so the image of $\gamma$ coincides with the graph $\Gamma(X)$. Consider the coordinate projections $\pi_x$ and $\pi_y$
	\[
	\xymatrixrowsep{.3in}
	\xymatrixcolsep{.6in}
	\xymatrix{
		& \Gamma(X) \ar@{->>}[dl]_-{\pi_x} \ar@{->>}[dr]^-{\pi_y} & \\
		\pp^n & & \pp^m 
	}
	\] onto the first and second factor, respectively. By construction, for each point $x$ in $X$ we have the equality $f(x)=\pi_y \circ \gamma(x)$; and since all spaces and maps in sight are projective, the images of closed sets remain closed. Now consider a subvariety $Y'\subset f(X)$ and note that its pre-image under $f$ is 
	\[
	f^{-1}(Y') := \pi_x\left(\Gamma(X) \cap \pi_y^{-1}(Y')\right).
	\] Treating $I_{Y'} \lhd \CC[y]$ as an ideal in $\CC[x,y]$, the desired pre-image $f^{-1}(Y')$ may be computed algebraically as the intersection 
	\[
	I_{f^{-1}(Y')} := \CC[x]\cap \left(I_{\Gamma(X)} +I_{Y'}\right).
	\]
	
	\subsection{Algorithm} 
	Assume that $f:X \to Y$ is a generic projective  morphism in the sense of Definition \ref{def:generic}, where $X\subset \pp^n$ and $Y\subset \pp^m$ are pure dimensional projective varieties with $f(X)\subset Y$, and let $k\leq \min(n+1,m+1)$ as in Section \ref{ssec:TBflag}. The following algorithm relies on {\bf WhitStratFlag} (from Sec \ref{ssec:flagstrat}) to build Whitney stratifications $X'_\bullet$ of $X$ and $Y'_\bullet$ of $Y$ via the following basic strategy. First we construct the image $\bB_\bullet$ of the Thom-Boardman flag $\bT_\bullet$ as described in Definition \ref{def:Bflag}. Next, we create a $\bB_\bullet$-subordinate stratification $Y_\bullet$ of the codomain $Y$ and pull it back across $f$ to obtain a flag $\bF_\bullet$ on the domain $X$. Finally, we create an $\bF_\bullet$-subordinate stratification $X'_\bullet$ of $X$.
	
	\medskip
	
	\begin{center}
		\begin{tabular}{|r|l|}
			\hline
			~ & {\bf WhitStratMap}$(X,Y,f)$ \\
			\hline
			~&{\bf Input:} Pure dimensional varieties $X,Y$ and 
			a generic morphism $f:X \to Y$.\\
			~&{\bf Output:} Lists of subvarieties $X_\bullet \subset X$ and $Y_\bullet \subset Y$. \\
			\hline
			~1 & {\bf Set} ${\bT_\bullet}X := (\bT_{k+1}X,\ldots,\bT_0X) := (X,\varnothing,\ldots,\varnothing)$ \\ 
			~2 & {\bf Set} ${\bB_\bullet}Y := (\bB_{k+2}Y,\ldots,\bB_0Y) := (Y,\varnothing,\ldots,\varnothing)$ \\
			~3 & {\bf For each} $j$ in $(0,1,\ldots,k)$ \\
			~4 & \spc {\bf Set } $\bT_jX := \bV(I_X + \textbf{Jac}_jf)$\\
			~5 & \spc {\bf Set } $\bB_jY := f(\bT_jX)$\\
			~6 & {\bf Set} $\bB_{k+1}Y := f(X)$ \\
			~7 & {\bf Set} $Y'_\bullet := {\bf WhitStratFlag}(Y,\bB_\bullet)$ \\
			~8 & {\bf For each} $i$ in $(0,1,\ldots,\dim Y)$ \\
			~9 & \spc {\bf Set} $\bF_iX := f^{-1}(Y'_i)$ \\
			10 & {\bf Set} $X'_\bullet := {\bf WhitStratFlag}(X,\bF_\bullet)$ \\
			11 & {\bf Set} $(X_\bullet,Y_\bullet) := ${ \bf  Refine}$(X'_\bullet,Y'_\bullet,f)$\\	
			12 & {\bf Return} $(X_\bullet,Y_\bullet)$\\
			\hline
		\end{tabular}
	\end{center}
	
	\medskip
	
	As we have not described the {\bf Refine} subroutine invoked in the penultimate line, we are not yet able to check whether this algorithm terminates, and whether it returns a correct stratification of $f$ in the sense of Definition \ref{def:stratmap3}. The next result, which involves the Whitney stratifications $X'_\bullet$ and $Y'_\bullet$ produced in Lines 10 and 7 respectively, will explain why this additional subroutine is needed.
	
	\begin{proposition}\label{prop:x'y'strat}
		For each stratum $S$ of $X'_\bullet$, there exists a unique stratum $R$ of $Y'_\bullet$ satisfying $f(S) \subset R$. Moreover, at each point $x$ in $S$, the Jacobian $J_xf|_S:T_xS \to T_{f(x)}R$ of the restricted map $f|_S:S \to R$ has full rank, i.e.,
		\[
		\text{\rm rank}\left(J_xf|_S\right) = \min\left(\dim S, \dim R\right)
		\]	
	\end{proposition}
	\begin{proof}
		Noting that of $X'_\bullet$ is subordinate to the flag $\bF_\bullet$ by Line 10, we know that for each stratum $S$ of $X'_\bullet$ there is a number $i := i(S)$ satisfying $S \subset (\bF_iX-\bF_{i-1}X)$. Now by Line 9, for any such stratum we have $f(S) \subset (Y'_i - Y'_{i-1})$. By Definition \ref{def:whitstrat} we know that $S$ is connected, and so its image under the continuous map $f$ must also be connected; thus there is a unique stratum $R \subset (Y'_i - Y'_{i-1})$ of $Y'_\bullet$ satisfying  $f(S) \subset R$. Since $Y'_\bullet$ is $\bB_\bullet$-subordinate by Line 7, there exists a number $j := j(R)$ satisfying $R \subset (\bB_jY - \bB_{j-1}Y)$. By Definition \ref{def:Bflag}, we have $f^{-1}(R) \subset (\bT_jX - \bT_{j-1}X)$, where $\bT_\bullet$ is the Thom-Boardman flag of $f$ from Definition \ref{def:Tflag}. In particular, this gives $S \subset (\bT_jX - \bT_{j-1}X)$ and it follows that $X'_\bullet$ is subordinate to $\bT_\bullet$. Consequently, for each $x \in S$ we know that $\text{rank }J_xf = (j-1)$. Consider the commuting diagram of vector spaces
		\[
		\xymatrixcolsep{.7in}
		\xymatrixrowsep{.35in}
		\xymatrix{
			T_xS \ar@{->}[r]^-{J_x(f|_S)} \ar@{_{(}->}[d] & T_{f(x)}R \ar@{^{(}->}[d] \\
			\CC^{n+1} \ar@{->}[r]_-{J_xf} & \CC^{m+1}
		}
		\]
		Here the vertical arrows depict inclusions of tangent spaces, e.g., on the left we have the natural inclusion of $T_xS$ in $T_x\pp^n \simeq \CC^{n+1}$. Since $S \subset (\bT_jX-\bT_{j-1}X)$, we know from \eqref{eq:tbrank} that the rank of $J_xf$ is precisely $(j-1)$ at every $x \in S$; and by genericity of $f$ it follows that the subspaces $T_xS$ and $\ker J_xf$, whose intersection equals $\ker J_xf|_S$, meet transversely inside $\CC^{n+1}$. Thus, we obtain 
		\begin{align*}
		\dim \ker(J_xf|_S) &= \dim \ker J_xf + \dim T_xS - (n+1) & \text{by transversality}\\ 
		&= [(n+1)-(j-1)] + \dim S - (n+1) & \text{by rank/nullity}\\
		&= \dim S - (j-1).
		\end{align*}
		There are now two cases to consider --- either $\dim S < (j-1)$, or $\dim S \geq (j-1)$. In the first case, $J_xf|_S$ is injective and hence already has full rank. In the latter case, we have $R \subset f(\Delta_j)$ where $\Delta_j := (\bT_jX-\bT_{j-1}X)$. Thus, $\dim Y \leq \dim f(\Delta_j)$; but since the rank of $f$ on $\Delta_j$ is $(j-1)$ by \eqref{eq:tbrank}, the implicit function theorem guarantees that $\dim f(\Delta_j) = (j-1)$. So, we have $\dim R \leq (j-1)$, and combining this inequality with our calculation of $\dim \ker (J_xf|_S)$ above gives
		\[
		\dim \ker(J_xf|_S) \leq \dim S - \dim R.
		\] 
		Since we have assumed $\dim S \geq \dim R$, the kernel of $J_xf|_S$ can not have dimension smaller than the codimension $\dim S - \dim R$, so the above inequality is an equality in this case and $J_xf|_S$ has full rank, as desired. 		
	\end{proof}
	
	When comparing the requirements of Definition \ref{def:stratmap3} to the properties guaranteed by the preceding result, we note that the stratifications $X'_\bullet$ and $Y'_\bullet$ are insufficient for our puposes. While every stratum $S$ of $X'_\bullet$ does indeed have a unique stratum $R$ of $Y'_\bullet$ containing $f(S)$, the crucial Jacobian-surjectivity requirement is not satisfied. Instead, we might have $\dim S < \dim R$ with the Jacobian $J_xf|_S$ being injective at every point $x$ in $S$. The {\bf Refine} subroutine invoked in Line 13 of {\bf WhitStratMap}, which we describe in the next Section, has been designed to rectify this defect.
	
	\subsection{Refinement and Correctness}
	
	Before Line 13 of {\bf WhitStrat} has been executed, we have stratifications $X'_\bullet$ and $Y'_\bullet$ of $X$ and $Y$ respectively. In light of Proposition \ref{prop:x'y'strat}, consider the set of problematic strata-pairs $\cP=\cP(X'_\bullet,Y'_\bullet)$ given by:
	\[
	\cP := \set{(S,R) \mid f(S) \subset R \text{ with } \dim S < \dim R}.
	\]
	The purpose of the {\bf Refine} subroutine described here is to modify the stratifications $X'_\bullet$ and $Y'_\bullet$ until this problematic set $\cP$ becomes empty. By Proposition \ref{prop:x'y'strat}, each stratum $S$ of $X'_\bullet$ can appear in a pair of $\cP$ with at most one stratum $R$ of $Y'_\bullet$. However, the converse need not hold --- a given stratum $R$ of $Y'_\bullet$ might be paired with several different strata of $X'_\bullet$ within $\cP$. Recalling that $X'_\bullet$ is $\bF_\bullet$-subordinate, for each index $\ell$ we will use $\fS_\ell = \fS_\ell(X'_\bullet)$ to denote the set of all strata $S$ of $X'_\bullet$ which lie in the difference $\bF_{\ell+1}X-\bF_{\ell}X$.
	
	\medskip
	
	\begin{center}
		\begin{tabular}	{|r|l|}
			\hline
			~ & {\bf Refine}$(X'_\bullet,Y'_\bullet,f)$ \\
			\hline
			~&{\bf Input:} Stratifications $X'_\bullet,Y'_\bullet$ of pure dimensional varieties $X$ and $Y$ \\ 
			& and a generic  morphism $f:X \to Y$ so that Prop. \ref{prop:x'y'strat} holds. \\
			~&{\bf Output:} Lists of subvarieties $X_\bullet \subset X$ and $Y_\bullet \subset Y$. \\
			\hline
			~1 & {\bf For each } $(S,R) \in \cP(X'_\bullet,Y'_\bullet)$ with $\dim R$ maximal \\
			~2 & \spc {\bf Set} $Y^+_\bullet := Y'_\bullet$ \\
			~3 & \spc {\bf Set} $d := \dim \overline{f(S)}$ \\
			~4 & \spc {\bf Add} $\overline{f(S)}$ to $Y^+_{\geq d}$ \\
			~5 &\spc {\bf Merge} $Y^+_\bullet$ with {\bf WhitStrat}$(\Pure_d(Y'_d))$\\
			~6 & \spc {\bf For each} $\ell = (d,d-1,\ldots,1,0)$ \\
			~7 & \spc \spc {\bf For each } irreducible $W \subset \overline{Y^+_\ell - Y'_\ell}$ and $S' \in \mathfrak{S}_\ell(X'_\bullet)$\\
			~8 &\spc \spc  \spc {\bf If } $Z \cap S' \neq \varnothing$ for an irreducible $Z \subset f^{-1}(W)$   \\
			~9 &  \spc  \spc  \spc \spc {\bf Set} $r := \dim Z$ \\
			10 &\spc  \spc  \spc\spc  {\bf Add} $Z$ to $X'_{\geq r}$\\
			11 &\spc  \spc  \spc\spc {\bf Merge} $X'_\bullet$ with {\bf WhitStrat}$({\rm Pure}_{r}(X'_{r}))$\\
			12 & \spc {\bf Set} $Y'_\bullet=Y^+_\bullet$\\
			13 & \spc {\bf Recompute} $\cP(X'_\bullet,Y'_\bullet)$\\
			14 & {\bf Return} $(X'_\bullet,Y'_\bullet)$ \\
			\hline
		\end{tabular}
	\end{center}
	
	\medskip
	
	This subroutine processes the problematic strata-pairs $(S,R)$ from $\cP(X'_\bullet,Y'_\bullet)$ one at a time, in descending order of $\dim R$. For each such pair, Lines 2-5 further partition $R$ by forcing the closure of the image $f(S)$ to form (one or more) new strata. As a result of subdividing $R$ along $f(S)$, some of the strata $S'$ of $X'_\bullet$ satisfying $f(S') \subset R$ no longer have their images contained in a single stratum. Lines 6-11 are designed to correct this problem by finding and further subdividing all such $S'$ appropriately. 
	
	\begin{proposition}
		The {\bf Refine} subroutine terminates, and its output $(X_\bullet,Y_\bullet)$ constitutes a valid stratification (as in Definition \ref{def:stratmap3}) of the generic projective morphism $f:X \to Y$.
	\end{proposition}
	\begin{proof}
		If $\cP$ is empty, then the algorithm terminates immediately with a correct stratification, so let $(S,R)$ be a strata-pair in $\cP$ with $R$ of maximal dimension; thus, we have $\dim S < \dim R$. In Lines 2-5, {\bf WhitStratMap} constructs a new Whitney stratification $Y^+_\bullet$ of $Y$ by subdividing the closure of $R$ into finitely many new strata
		\[
		\overline{R}= \coprod_i R_i
		\]	
		so that all points lying in the (closure of the) immersed image $f(S) \subset R$ lie within strata of dimension no larger than $\dim S$. Thus, for each point $x \in S$, there is a unique index $i(x)$ for which $f(x) \in R_{i(x)}$. Moreover, we have
		\[
		\dim R_{i(x)} \leq \dim S < \dim R.
		\]
		Since $R_{i(x)} \subset R$, we know by Proposition \ref{prop:x'y'strat} that the Jacobian $J_xf:T_xS \to T_{f(x)}R_{i(x)}$ is surjective for all $x \in S$ as desired. Unfortunately, the act of partitioning $R$ into smaller strata might violate the other property required by Definition \ref{def:stratmap3}, i.e., we may have strata $S'$ of $X'_\bullet$ whose images $f(S')$ were entirely contained in $R$, but which now intersect several new $R_i$'s. Therefore, in Lines 6-11, the algorithm paritions all such $S'$ along their intersections with $f^{-1}(R_i)$ for all $i$ --- this creates new Whitney strata and hence refines $X'_\bullet$. Finally, in Line 12 we also update $Y'_\bullet$ to the new stratification $Y^+_\bullet$. After this update, there might be several new problematic strata pairs in $\cP$; but the key observation here is that none of these new pairs $(S^*,R^*)$ can have $\dim R^* > \dim R$ since all of the new strata $R^*$ of $Y'_\bullet$ have dimension bounded above by $\dim R$. Moreover, even when $\dim R^* = \dim R$, it is impossible to have any $(S^*,R^*)$ in $\cP$ where $S^* \subset S$ is a newly-created stratum of $X'_\bullet$ --- any such $S'$ must have its image $f(S')$ entirely contained in a stratum of dimension $< \dim R$. Thus, {\bf WhitStratMap} eventually terminates and outputs a valid stratification of $f:X \to Y$.
	\end{proof}
	
	\begin{remark} An analogous algorithm for stratifying proper algebraic maps between affine varieties can be produced by making a few standard modifications to {\bf WhitStratMap}. Implicit in (the proofs of) Corollaries \ref{cor:affine_Whit_Strat} and \ref{cor:staratWrtFlagAffine} are routines which would compute (flag-subordinate) Whitney stratifications of a complex affine algebraic variety. These routines accept as input affine equations, homogenize the associated Gr\"obner basis, run the projective stratification algorithms described here, and then dehomogenize the output to produce the final result. Let {\bf WhitStratAff} and {\bf WhitStratFlagAff} be the resulting algorithms. If $X\subset \CC^n$ and $Y\subset \CC^m$ are affine varieties and $f:X\to Y$ is a proper morphism, then an analogous definition of the map being generic and of a Thom-Boardman flag on $X$ may be given.  To obtain a stratification of the generic proper morphism $f:X\to Y$ in the sense of Definition \ref{def:stratmap3}, one simply replaces each occurence of {\bf WhitStrat} and {\bf WhitStratFlag} by {\bf WhitStratAff} and {\bf WhitStratFlagAff}, respectively in the {\bf WhitStratMap} and {\bf Refine} algorithms described above.  
	\end{remark}
	
	\section{Performance and Complexity} \label{sec:runtimes}
	
	In this section, we briefly describe the real-life performance as well as the computational complexity of the basic algorithm {\bf WhitStrat} from \S\ref{sec:WSalgorithm}. Our implementation is in Macaulay2 \cite{M2}, and can be found along with documentation at the link below:	\begin{center}
		\url{http://martin-helmer.com/Software/WhitStrat}
	\end{center} 
	
	\subsection{Performance}

	As remarked in the Introduction, the state of the art for Whitney stratification algorithms appears to be the recent algorithm of {\DD}inh and Jelonek \cite[\S2]{dhinh2019thom}. We are not aware of any existing implementations of earlier algorithms based on quantifier elimination, and in any event, since implemented quantifier elimination methods rely on cylindrical algebraic decomposition we would expect their performance to be slower than that of  \cite{dhinh2019thom} overall. Since the authors of \cite{dhinh2019thom} did not provide an implementation of their algorithm, we have implemented it ourselves in Macaulay2. This implementation can be found at:
	\[  \text{\url{http://martin-helmer.com/Software/WhitStrat/DCG.m2}.}
	\] 
	 In Table \ref{table:runtimes} we show the run time of the {\bf WhitStrat} on several examples; on all of the examples listed here, the algorithm of \cite{dhinh2019thom} did not finish after 8 hours of calculation time. 
	
		\begin{table}[h!]\resizebox{.85\linewidth}{!}{
			\renewcommand{\arraystretch}{1.5}
			\begin{tabular}{@{} l *8c @{}}
				\toprule 
				\multicolumn{1}{l}{{\textbf{INPUT}}}    &\quad\quad {\bf Run time} \quad   \quad \\ 
				\midrule 
				$\bV\left(x_0x_1^2-x_1^2x_2\right) \subset \pp^3$ & 0.2s \\ 
				$\bV\left(	{x}_{1}^{4}{x}_{2}-{x}_{0}^{5}-{x}_{0}^{4}{x}_{3}-{x}_{0}^{4}{x}_{4}\right)\subset \pp^4$ & 0.6s \\
				$\bV\left({x}_{3}^{3}-{x}_{1}{x}_{2}^{2}-{x}_{0}^{2}{x}_{3}+{x}_{0}^{2}{x}_{4}-{x}_{3}{x}_{4}^{2}\right)\subset \pp^4$ & 0.9s \\
				$\bV\left({x}_{6}^{2}-{x}_{1}{x}_{2}+{x}_{0}{x}_{4},{x}_{0}^{2}-{x}_{0}{x}_{3}-{x}_{5}^{2}\right)\subset \pp^6$& 1.7s\\
				$\bV\left({x}_{0}^{2}{x}_{4}-{x}_{1}{x}_{2}^{2}+{x}_{3}^{3},{x}_{0}^{2}-{x}_{1}{x}_{4}\right)\subset \pp^4$&  2.0s\\
				$\bV\left({x}_{4}{x}_{7}-{x}_{1}{x}_{2}+{x}_{7}^{2},{x}_{0}^{2}-{x}_{0}{x}_{5}-{x}_{7}^{2},{x}_{3}{x}_{7}-{x}_{6}^{2}\right)\subset \pp^7$& 272.8s \\
				\bottomrule
		\end{tabular}} \caption{Run times of our {\bf WhitStrat} implementation in Macaulay2 when working over $\mathbb{Q}$ on an Intel Xeon E3-1245 v5 CPU with 32 GB of RAM. The algorithm of \cite{dhinh2019thom} did not finish after 8 hours of calculation time on these examples (we also tried the algorithm of \cite{dhinh2019thom} on a computer with a Intel i7-8700 CPU and 64 GB RAM, and again it did not finish within 8 hours). 
			\label{table:runtimes}}
	\end{table}
	
	The variety in the first entry of this table is (the projective analogue of) the Whitney umbrella, which was one of the earliest examples of singular spaces used to illustrate the need for Condition (B). We ran the algorithm of \cite{dhinh2019thom} on this variety for over 29 hours. In this time, it was unable to find strata of codimension $> 1$, i.e., the strata lying below the singular locus. All computations involving this algorithm used at least 27 GB of RAM during their 8 hour run. In sharp contrast, {\bf WhitStrat} used between 0.0005 and 0.347 GB of RAM, with all but the last entry in the table requiring no more than 0.0037 GB. While it is certainly possible that a more optimized implementation of \cite{dhinh2019thom} than ours could be produced, in our view it is unlikely to significantly alter these general trends.
	
	\subsection{Complexity}
	
	By far the biggest computational burden incurred when running {\bf WhitStrat} comes from various intermediate Gr\"{o}bner basis calculations. In particular, these are required when running the {\bf Decompose} subroutine in Line 3 (computing all primary components), and in Line 4 (computing elimination ideals), for more on how these tasks are accomplished via Gr\"{o}bner basis  see \cite{Cox,decker1999primary}. Throughout the remainder of this section, we write
	$\textbf{GB}(n,\delta)$ to denote the complexity of performing Gr\"obner basis computation on an ideal consisting of polynomials in $n$ variables of degree at most $\delta$, in the worst case this bound is approximately $\mathcal{O}(\delta^{2^n})$, as described for instance in \cite{mayr1982complexity,moller1984upper}. However, more recent analysis \cite{bardet2015complexity,faugere1999new} has revealed that this worst case bound tends to be overly pessimistic in general. 
	
	The bulk of our complexity estimates will be provided in terms of the following input parameters: we assume that the defining radical ideal $I_X$ of our input variety $X$ has been given in terms of a finite generating set of homogeneous polynomials in $n$ complex variables with maximal degree $\delta$. To avoid trivialities we assume $\delta\geq 2$ and $n\geq 2$. Let $k$ be the dimension of $X$, let $c = n-k$ be the codimension of $X$, and $\mu < k$ the dimension of the singular locus $X_\text{sing}$.

\begin{proposition}\label{prop:basecase}
    The degree of $X_\text{\rm sing}$ is no larger than $\delta^{n}$. If $Z \subset X_\text{\rm sing}$ is any irreducible $\mu$-dimensional subvariety, then the degree of the affine subvariety $V_\mu \subset \CC^n \times \CC^n$ defined by the ideal sum $I_{\Con(X)}+I_Z$ is bounded as
    \[
    \deg(V_\mu) \leq \delta^{3n^2}.
    \]
\end{proposition}
\begin{proof}
We begin with the observation that $\deg(X_\text{sing})$ is bounded from above by the degree of any $\mu$-dimensional complete intersection that contains $X_\text{sing}$ --- such complete intersections may be generated as zero sets of an appropriate number of random linear combinations of the defining polynomials of $X_\text{sing}$ (see \cite[Example 8.4.12]{fulton2013intersection} or \cite[\S A.9]{wampler2005numerical}), after choosing a set of generators all having the same degree (which can always be done for homogeneous ideals \cite[Remark~2.3]{HH19}). Recall that the Jacobian minors which generate $X_\text{sing}$ in the coordinate ring of $X$ have degree at most $\delta-1$. By B\'ezout's theorem \cite[Theorem~12.3]{fulton2013intersection} and the discussion above, we have
\[
\deg(X_\text{sing}) \leq \deg(X) \cdot (\delta-1)^{k},
\]
since the singular locus can have codimension at most $k$ in $X$. Again using Bezout's theorem, the expression on the right may be bounded above by $\delta^{n-k} \cdot \delta^k$, which equals $\delta^{n}$ as desired. Turning now to $V_\mu$, note that the minors of the augmented Jacobian which generate $\Con(X)$ have bidegree no larger than $((\delta-1)^{n-k},1)$ inside $\PP^n \times (\PP^n)^*$, while those which generate $Z$ have degree no larger than $\delta^{n}$ inside $\PP^n$ by our bound on $\deg(X_\text{sing})$ given above. Thus, $(\delta^{n},1)$ bounds from above the  bidegree for the generating polynomials of $V_\mu$. Passing to affine charts, $V_\mu$ may be generated by polynomials of degree at most $\delta^{n}+1$ as a subvariety of $\CC^{2n}$. Appealing once again to a complete intersection which contains $V_\mu$ and has the same dimension, we obtain the following degree bound by B\'ezout's theorem: 
\[
\deg(V_\mu) \leq (\delta^{n}+1)^{\codim V_\mu} <  (\delta^{n}+1)^{2n}\leq (2\delta^{n})^{2n}\leq (\delta^{n+1})^{2n}. 
\]
Since the right side is bounded from above by $\delta^{3n^2}$, the argument is complete.
\end{proof}

The degree of $V_\mu$ estimated above bounds several quantities of interest from above when {\bf Decompose} is first invoked from {\bf WhitStrat}. In {\bf Decompose}, we loop over all primary components of an ideal. Note that the desired quantities, i.e.~degrees and number of generators of the resulting ideals, are bounded by an expression which is polynomial in $\delta^{2^n}$. In Proposition \ref{prop:decompcost} below we are led to consider various ideals in $\CC[x_1,\dots, x_n]$ and $\CC[x_1, \dots, x_n,\xi_1,\dots, \xi_n]$. For an arbitrary  polynomial ideal $J$  with generators of degree at most $d$ in $N$   variables (we will have either $N=n$ or $N=2n$) we therefore bound both the degree and number of primary components of $J$ by $d^{\ell 2^N}$, for some fixed positive integer $\ell$ that is independent of the input. An analysis of the primary decomposition algorithm of \cite{gianni1988grobner}, see also \cite{decker1999primary}, yields a (not sharp) upper bound $\ell \leq 36N^2$.

\begin{proposition}\label{prop:decompcost}
    The following quantities are bounded by $$\delta^{(3\ell)^\mu n^{2\mu}2^{n\mu}}$$ when {\bf Decompose} is called from {\bf WhitStrat} during the any step of the algorithm:
    \begin{enumerate}
        \item the number of primary components of $J$;
        \item the degrees of the generating polynomials of $J$ and $K$.
    \end{enumerate}
\end{proposition}

 \begin{proof}
 We have from Proposition 8.1 that $\deg(V_\mu)$ is bounded by $\delta^{3n^2}$. For all $d$ in $\set{\mu-1,\ldots,0}$ we similarly let $V_d$ be the degree of the ideal $J=I_{\Con(X)}+I_Y$ in Line 2 of {\bf Decompose} when it is called during the $(\mu-d+1)$-st iteration of the {\bf For} loop of {\bf WhitStrat}. Thus, when $d=\mu-1$, the variety represented by $X$ in Line 3 of Decompose is $\mu$-dimensional and lives in the singular locus of the original input variety, whereas $Z \subset X$ has dimension at most $\mu-1$.

Starting from $Y=X_{\rm Sing}$ we may call {\bf Decompose} upto $\mu$ times. In the first instance we run line with an input of degree at most $\delta^{3n^2}$, this will yield up to $$(\delta^{3n^2})^{\ell 2^n}=\delta^{3\ell n^22^n}$$ primary components, whose degree is bounded by this same number and this number also bounds the degree of $K$ since the ideal elimination cannot increase degree. Each iteration will increase the bound on the degree and number of components by exponentiating to ${3n^22^n}$ the bounds from the last iteration, this can happen to $\delta^{3\ell n^22^n}$ at most $\mu-1$ additional times, so we obtain $$\delta^{(3\ell n^22^n)^{\mu}}.$$

 A very similar argument shows that this quantity also bounds the degrees of varieties arising in {\bf Decompose} from the recursive calls in Line 7 of {\bf WhitStrat}. \qedhere
 \end{proof}

In order to fully describe {\bf WhitStrat}'s complexity, we require the following additional {\em runtime parameters}. Assume, for each $d \in \set{\mu,\mu-1,\ldots,1,0}$, that $r_d \geq 0$ is the largest number of generators encountered among the ideals $J,K$, and $\sqrt{K}$ (which is the ideal of $\bV(K)$) in Lines 3-5 of {\bf Decompose} when extracting the $d$-dimensional strata of $X$ (which are then used in subsequent conormal variety calculations). We have introduced these secondary $r_d$ parameters because we are unaware of any sufficiently general and reasonably tight bounds on the number of generators of $J,K$, and $\sqrt{K}$  in terms of the primal quantities $n$, $\delta$ and the number of generating polynomials for the input variety $X$. Any such bound must account for pathological cases  and hence be at least doubly exponential in $n$, see for instance \cite{caviglia2021bounds}. Such estimates do not accurately reflect the vast majority of inputs. 

If a given ideal $I_Z$ has $r$ generating polynomials in $n$ variables, then the number of minors in the generating ideal of $\Con(Z)$ equals
\[
\binom{r+1}{n-\dim(Z)}\cdot \binom{n+1}{n-\dim(Z)} \leq (16 r n )^{n-\dim(Z)}<(16 r n )^n.
\] We will assume henceforth that $(16 r_d n )^n$ is in $\mathcal{O}(\delta^{(3\ell)^\mu n^{2\mu}2^{n\mu}})$ for all $d$, so that the cost of computing the required minors to obtain the ideals of conormal varieties is dominated by the cost of subsequent Gr\"obner basis calculations. With this assumption in place, the estimates in Proposition \ref{prop:decompcost} give us the desired complexity bounds. 

\begin{theorem}\label{cor:total_Complexity}
Let $X$ be a complex projective subvariety of $\pp^n$ whose defining polynomials have degree at most $\delta$ and whose singular locus has dimension $\mu$. The time complexity of running {\bf WhitStrat} on $X$ is bounded in
\[
    \mathcal{O}\left((\mu+2)^2\cdot (D+2)\cdot \text{\bf GB}(n,D)\right),
\] where $D = \delta^{(3\ell)^\mu n^{2\mu}2^{n\mu}}$ and $\text{\bf GB}(n,D)$ is the cost of computing Gr\"obner bases for an ideal generated by polynomials in $n$ variables of degree at most $D$.
\end{theorem}

\begin{proof}

 Taking into account the recursive call in Line 7 of {\bf WhitStrat}, the subroutine {\bf Decompose} can be called at most $(\mu+1)+(\mu)+\dots + 1=\frac{(\mu+1)(\mu+2)}{2}$ times. Each call results in a Gr\"obner basis computation at Line 2 to find $I_{\Con}(X)$ and the computation of a primary decomposition prior to the loop executing on Line 3; the computational complexity of this primary decomposition is polynomial in the cost of Gr\"obner basis computation as discussed above.  Each iteration of the {\bf For} loop in Line 3 of {\bf Decompose} runs over at most $D$ primary components and computes an elimination ideal in Line 4 requiring one Gr\"obner basis computation; the computation of the dimension at Line 5 can be done using linear algebra operations using the output of Line 4 and hence has complexity bounded by the operation at Line 4.  Hence we have at most $D+2$ operations each have complexity in at most the order of $\text{\bf GB}(n,D)$. Bounding $\frac{(\mu+1)(\mu+2)}{2}$ by $(\mu+2)^2$ gives the claimed result. 
 \end{proof}

The practical efficiency of {\bf WhitStrat} comes from the fact that the parameter $D$ from the above theorem lies in $\mathcal{O}(\deg(X))$ for typical input varieties $X$, rather than scaling super-exponentially with $\delta$. For the purposes of describing theoretical complexity, we are forced to increment the power of $n$ in the exponent each time the {\bf For} loop of {\bf WhitStrat} iterates, as this is the only way to guarantee that the worst-case degree bounds hold. In practice, however, we have observed that the degrees of the intermediate varieties stay roughly constant and even reduce as the {\bf For} loop iterates from $d = \mu$ down to $d = 0$.

\subsection{Comparison with Quantifier-Elimination based Stratification Methods}\label{subsec:comparison_theory}
	
	We are aware of one other approach for computing Whitney stratifications which does not use Gr\"obner basis calculation, namely the algorithm of \cite{rannou1991complexity,rannou1998complexity}. This method has complexity bounds derived from the asymptotically fast critical points based quantifier elimination methods of Grigoriev-Vorobjov \cite{grigor1988solving} and Renegar \cite{renegar1992computational}. Although these critical point methods were first introduced over three decades ago and have been extensively studied thereafter, to the best of our knowledge there is no software implementation for general polynomial systems. All general purpose quantifier-elimination software we are aware of uses some form of cylindrical algebraic decomposition, \cite{arnon1984cylindrical,Basu2006Algorithms,caviness2012quantifier}, which has complexity bounds doubly exponential in the number of variables.  Quantifier-elimination algorithms based on critical points have been described in \cite[Chapter 14]{Basu2006Algorithms}; on page 7 of this text, the authors state that:
	\begin{quote}
	    {\em ``For most of the algorithms presented in Chapter 13 to 16 there is no implementation at all. The reason for that is that the methods developed are well adapted to complexity results but are not adapted to efficient implementation."}
	\end{quote}
	
    Nevertheless, had an implementation of the critical points based quantifier-elimination  methods of Grigoriev-Vorobjov \cite{grigor1988solving} and Renegar \cite{renegar1992computational} existed, then an implementation of the stratification algorithm of  \cite{rannou1991complexity,rannou1998complexity} based on this would achieve a complexity bound of the form
    \[
    \delta^{\mathcal{O}(n)^{6 \omega}},
    \] where $\omega$ is the length of the canonical (i.e., coarsest) Whitney stratification of the input variety, see, for example, \cite[page~289]{rannou1991complexity}. For high dimensional inputs which admit strata in all dimensions, this bound is asymptotically quite similar to those from Gr\"{o}bner basis based methods. However, it has been suggested that the $\mathcal{O}(n)$ in the exponent conceals a rather large constant that is likely to make these quantifier-elimination methods infeasible in practice \cite[Remark 2.28]{basu2014algorithms}.

	\bibliographystyle{abbrv}
	\bibliography{library}
\end{document}